\documentclass[reqno]{amsart}
\usepackage{color,graphicx,wrapfig}
\title[]{A geometric approach to generalized Stokes conjectures}
\dedicatory{Dedicated to John Toland on the occasion of his 60th birthday}
\author{Eugen Varvaruca}
\address{Department of Mathematics, Imperial College London, 180 Queen's Gate,
London SW7 2AZ, United Kingdom}
\email{e.varvaruca@imperial.ac.uk}
\author[Georg S. Weiss]{Georg S. Weiss}
\address{Graduate School of Mathematical Sciences,
University of Tokyo, 3-8-1 Komaba, Meguro-ku, Tokyo-to, 153-8914
Japan,} \email{gw@ms.u-tokyo.ac.jp}
\urladdr{http://www.ms.u-tokyo.ac.jp/~gw/}
\thanks{$2000$ {\it Mathematics Subject Classification.\/} Primary
35R35, Secondary 76B15, 76B07.}
\thanks{G.S. Weiss has been partially supported by the Grant-in-Aid
15740100/18740086 of the Japanese Ministry of Education, Culture,
Sports, Science and Technology.}
\date{}

\theoremstyle{plain}
\newtheorem{theorem}{Theorem}[section]
\newtheorem{lemma}[theorem]{Lemma}
\newtheorem*{theorema}{Theorem A}
\newtheorem*{theoremb}{Theorem B}
\newtheorem*{acknowledgment}{Acknowledgment}
\newtheorem{proposition}[theorem]{Proposition}
\newtheorem{corollary}[theorem]{Corollary}
\theoremstyle{definition}
\newtheorem{definition}[theorem]{Definition}
\theoremstyle{definition}
\newtheorem{remark}[theorem]{Remark}
\numberwithin{equation}{section}

\def\R{{\bf R}}

\def\W{{\bf W}}
\def\div{\textrm{\rm div }}

\def\dist{\textrm{\rm dist}}

\def\dh{\,d{\mathcal H}^{n-1}}
\def\dhone{\,d{\mathcal H}^1}

\newcommand{\Om}{\Omega}

\newcommand{\non}{\nonumber}
\newcommand{\be}{\begin{equation}}
\newcommand{\ee}{\end{equation}}

\begin{document}
\begin{abstract}
We consider the Stokes conjecture concerning the shape of extreme
two-dimensional water waves. By new geometric methods including a nonlinear
frequency formula, we prove
the Stokes conjecture in the original variables. Our results do not
rely on structural assumptions needed in previous results such as isolated
singularities, symmetry and monotonicity. Part of our results extends
to the mathematical problem in higher dimensions.
\end{abstract}
\maketitle
\tableofcontents
\section{Introduction}

Consider a two-dimensional inviscid incompressible fluid acted on by
gravity and with a free surface. If we denote by $D(t)\subset
\R^2$ the domain occupied by the  fluid at time $t$, then
the dynamics of the fluid is described by the Euler equations
for the vector velocity field $(u(t,\cdot),v(t,\cdot)):D(t)\to
\R^2$ and the scalar pressure field
$P(t,\cdot):D(t)\to\R:$ \begin{align}u_t + u u_x + vu_y
&=-P_x\quad\text{in }D(t),\non\\
v_t + u v_x + vv_y &=-P_y-g\quad\text{in
}D(t),\non\\u_x+v_y&=0\quad\text{in }D(t),\non
\end{align}
where subscripts denote partial derivatives and $g$ is the gravity
constant. The boundary $\partial D(t)$ of the fluid domain contains
a part, denoted by $\partial_a D(t)$, which is free and in contact
with the air region.
 The equations of motion are supplemented by the standard kinematic
boundary condition
\[V=(u,v)\cdot \nu \quad\text{on }\partial_a D(t),\] where $V$ is the normal speed of $\partial_a D(t)$
 and  $\nu$ is the outer normal vector, and the dynamic boundary
condition
\[P\text{ is locally constant on } \partial_a D(t).\]
We further assume that the flow is irrotational:
\[u_y-v_x=0\quad\text{in }D(t).\]

While recent years have seen great progress in the study of the
initial-value problem (see \cite{swu} for large-time well-posedness
for small data, and the references therein for short-time
well-posedness for arbitrary data), in the present paper  we confine ourselves to
traveling-wave solutions of the above problem, for which there
  exists $D\subset\R^2$, $c\in \R$, $(\tilde u,\tilde v): D\to\R^2$ and
$\tilde P:D\to\R$ such that
\[ D(t)=D+ct(1,0)\quad\text{for all }t\in\R,\]
and for all $t\in\R$ and $(x,y)\in D(t)$,
\[ u(x,y,t)=\tilde u(x-ct,y)+c,\quad v(x,y,t)=\tilde v(x-ct,y),\quad P(x,y,t)=\tilde P(x-ct,y). \]
Consequently the following equations are satisfied:
\begin{align}\tilde u
\tilde u_x + \tilde v\tilde u_y& =-\tilde P_x\quad\text{in }D,\non\\
\tilde u \tilde v_x + \tilde v\tilde v_y &=-\tilde P_y-g
\quad\text{in
}D,\non\\
\tilde u_x+\tilde v_y&=0 \quad\text{in }D,\non\\ \tilde u_y
-\tilde
v_x&=0 \quad\text{in }D,\non\\
(\tilde u,\tilde v)\cdot \nu&=0 \quad\text{on }\partial_a D\non,\\
\tilde P\text{ is locally }&\text{constant on } \partial_a D.\non
\end{align}

The above problem describes both {\em water waves}, in which case we
would add homogeneous Neumann boundary conditions on a flat horizontal
bottom $y=-d$ combined with periodicity in the $x$-direction or some condition at
$x=\pm \infty$, and the equally physical problem of the equilibrium
state of a fluid when pumping in water from one lateral boundary and
sucking it out at the other lateral boundary. In the latter setting
we would consider a bounded domain with an inhomogeneous Neumann
boundary condition at the lateral boundary, and the bottom could be
a non-flat surface.

In both cases, the incompressibility  and the kinematic boundary condition
 imply that
there exists a {\em stream function} $\psi$ in $D$, defined up to a
constant by:
\[ \psi_x=-\tilde v,\qquad\psi_y=\tilde u\qquad\text{in }D.\]
It follows that \[\psi\text{ is locally constant on }\partial_a D.\]
The irrotationality condition shows that
\[\psi \text{ is a harmonic function in } D,\]
and then Bernoulli's principle gives that
\[ \tilde P +\frac{1}{2}|\nabla \psi|^2+gy\quad\text{is constant in } D.\]
 The dynamic boundary condition implies therefore the {\em Bernoulli condition}
\[|\nabla \psi|^2 +2gy\quad\text{is locally constant on } \partial_a
D.\]

A \emph{stagnation point} is one at which the relative velocity
field $(\tilde u, \tilde v)$ is zero, and a wave with stagnation
points on the free surface will be referred to as an \emph{extreme
wave}. Consideration of extreme waves goes back to Stokes, who in
1880 made the famous conjecture that the free surface of an extreme
wave is not smooth at a stagnation point, but has symmetric lateral
tangents forming an angle of $120^\circ$. Stokes \cite{stokespaper}
gave a formal argument in support of his conjecture, which can be
found at the end of this introduction, but a rigorous proof has not
been given until 1982, when Amick, Fraenkel and Toland \cite{toland}
and Plotnikov \cite{plotnikov} proved the conjecture independently
in brilliant papers. These proofs use an equivalent formulation of
the problem as a nonlinear singular integral equation due to
Nekrasov (derived via conformal mapping), and are based on rather
formidable estimates for this equation. In addition, Plotnikov's
proof uses ordinary differential equations in the complex plane.
Moreover, Plotnikov and Toland proved convexity of the two branches
of the free surface \cite{convex}. Prior to these works on the
Stokes conjecture, the existence of extreme periodic waves, of
finite and infinite depth, had been established by Toland \cite{To}
and McLeod \cite{McL}, building on earlier existence results for
large-amplitude smooth waves by Krasovskii \cite{Kr} and by Keady
and Norbury \cite{KN}. Also, the existence of large-amplitude smooth
solitary waves and of extreme solitary waves had been shown by Amick
and Toland \cite{AT}.

In the present paper we
confine ourselves to the case when
\begin{align}
\Delta \psi&=0\quad\text{in }D,\non\\
 \psi&=0\quad\text{on }\partial_a D,\non\\
 |\nabla \psi(x,y)|^2&=-y\quad\text{on  }\partial_a D,\non
\end{align}
and we investigate the shape of the free surface $\partial_a D$
close to stagnation points for extreme waves which {\em a priori}
satisfy minimal regularity assumptions. Note that, since $(\tilde
u,\tilde v)=(\psi_y,-\psi_x)$, the Bernoulli condition implies that
 the free surface is contained in the lower half-plane  and
 that the stagnation points
on the free surface necessarily lie on the real axis and are points of maximal height.

Weak solutions of the above free-boundary problem have been studied
by Shargorodsky and Toland \cite{ST} and Varvaruca \cite{EV3}, who
consider solutions for which the free surface $\partial_a D$ is a
locally rectifiable curve, $\psi\in C^2(D)\cap C^0(\overline D)$ is
harmonic and satisfies the zero Dirichlet boundary condition in the
classical sense, while the Bernoulli condition is satisfied almost
everywhere with respect to the one-dimensional Hausdorff measure by
the non-tangential limits of $\nabla\psi$. They prove that the set
$S$ of stagnation points on the free surface is a set of zero
one-dimensional Hausdorff measure, that $\partial_a D\setminus S$ is
a union of real-analytic arcs, and that $\psi$ has a harmonic
extension across $\partial_a D\setminus S$ which satisfies all
free-boundary conditions in the classical sense outside stagnation
points.

The main objectives of
the present paper are to give affirmative answers to the following two
questions:
\begin{itemize}
\item[(i)] Does the set $S$ consist only of
isolated points?
\item[(ii)] Is the Stokes Conjecture valid at each point of
$S$ ?
\end{itemize}
Prior to our work, Question (i) has been completely open, while the
answer to Question (ii) has been known only partially: from
\cite{toland} and \cite{plotnikov} which have recently been
simplified in \cite{EV} and \cite{EV2}, we know (ii) to be true at
those points of $S$ which satisfy the following conditions in a
neighborhood of the stagnation point: The stagnation point is
isolated, the free surface is symmetric with respect to the vertical
line passing through the stagnation point, it is a monotone graph on
each side of that point, and $\psi$ is strictly decreasing in the
$y$-direction in $D$. All of these conditions are essential for the
proofs in the cited results. Let us mention that from the point of
view of applications, the requirement of symmetry is most
inconvenient, as numerical results indicate the existence of
nonsymmetric extreme waves \cite{chen,VdB,zuf}. Also, for waves with
non-zero vorticity, $\psi$ need not be monotone in the $y$-direction
\cite{W, CV}.

Similarly to \cite{ST} and \cite{EV3}, we consider weak solutions
which are roughly speaking solutions in the sense of distributions.
The precise notion will be given in Definition \ref{weak}. We assume
that $\psi >0$ in $D$, and we extend $\psi$ by the value $0$ to
the air region  so that the fluid domain can be identified with the
set $\{\psi>0\}$. Since our arguments are local, we work in a
bounded domain $\Omega$ which has a non-empty intersection with the
real axis and on which there is defined a continuous function $\psi$
such that, within $\Omega$, $\{\psi>0\}$ corresponds to the fluid
region and $\{\psi=0\}$ to the air region, the part of $\Omega$ in
the upper half-plane being occupied by air.

In the case of only a finite number of connected components of the
air region we recover the Stokes conjecture by geometric methods
(Theorem \ref{main2}), without assuming isolatedness, symmetry or
any monotonicity:

\begin{theorema}
Let $\psi$ be a weak solution of
\begin{align}
 \Delta \psi &= 0 \quad\text{in } \Omega \cap \{\psi>0\},\non\\
{\vert \nabla \psi\vert}^2
  &=   -y  \quad\text{on } \Omega\cap\partial
\{\psi>0\},\non
\end{align}
 and suppose that
 \[|\nabla \psi|^2 \leq -y\quad\text{in }\Omega \cap \{\psi>0\}.\]
Suppose moreover that $\{ \psi=0\}$ has locally only finitely many
connected components. Then the set $S$ of stagnation points is
locally in $\Omega$ a finite set. At each stagnation point
$(x^0,y^0)$ the scaled solution converges to the Stokes corner flow,
that is,
\[\frac{\psi((x^0,y^0)+r(x,y))}{r^{3/2}}\to
\frac{\sqrt{2}}{3} \rho^{3/2}\max(\cos(3(\theta+\pi/2)/2),0)
\quad\text{as } r\searrow 0,\] strongly in
$W^{1,2}_{\textnormal{loc}}(\R^2)$ and locally uniformly on $\R^2$,
where $(x,y)=(\rho\cos\theta, \rho\sin\theta)$, and in an open
neighborhood of $(x^0,y^0)$ the topological free boundary $\partial
\{\psi>0\}$ is the union of two $C^1$-graphs with right and left
tangents at $(x^0, y^0)$.
\end{theorema}

Let us remark that the assumption
 \[|\nabla \psi|^2 \leq -y\quad\text{in }\{\psi>0\}\]
has been verified in \cite[Proof of Theorem 3.6]{EV3} for weak
solutions, in the sense of \cite{ST} and \cite{EV3} described
earlier, of the water-problem in all its classical versions:
periodic and solitary waves of finite depth (in which the fluid
domain has a fixed flat bottom $y=-d$, at which $\psi$ is constant),
and periodic waves of infinite depth (in which the fluid domain
extends to $y=-\infty$ and the condition $\lim_{y\to -\infty}
\nabla\psi (x,y)=(0,-c)$ holds, where $c$ is the speed of the wave).
The proof is merely an extension of that of Spielvogel \cite[Proof
of Theorem 3b]{spiel} for classical solutions, which is based on the
Bernstein technique.

In the case of an infinite number of connected components of the air
region, we obtain the following result (cf. Theorem \ref{main1}):

\begin{theoremb}
Let $\psi$ be a weak solution of
\begin{align}
 \Delta \psi& = 0 \quad\text{in } \Omega \cap \{\psi>0\},\non\\
{\vert \nabla \psi\vert}^2
  &=   -y  \quad\text{on } \Omega\cap\partial
\{\psi>0\},\non
\end{align}
 and suppose that \[|\nabla \psi|^2 \leq -y\quad\text{in }\Omega\cap \{\psi>0\}.\]
Then the set $S$ of stagnation points is a finite or countable set.
Each accumulation point of $S$ is a point of the locally finite set
$\Sigma$ described in more detail in the following lines.

At each point $(x^0,y^0)$ of $S\setminus\Sigma$,
\[\frac{\psi((x^0,y^0)+r(x,y))}{r^{3/2}}\to \frac{\sqrt{2}}{3}
\rho^{3/2}\max(\cos(3(\theta+\pi/2)/2),0) \quad\text{as } r\searrow
0,\] strongly in $W^{1,2}_{\textnormal{loc}}(\R^2)$ and locally
uniformly on $\R^2$, where $(x,y)=(\rho\cos\theta, \rho\sin\theta)$.
 The
scaled free surface converges to that of the Stokes corner flow in
the sense that,  as  $r\searrow 0$, \[\mathcal{L}^2(B_1\cap
(\{(x,y):
\psi((x^0,y^0)+r(x,y))>0\}\bigtriangleup\{\cos(3(\theta+\pi/2)/2)>0\}))
\to 0.\]

At each point $(x^0,y^0)$ of $\Sigma$ there exists an integer $N=
N(x^0,y^0)\geq 2$ such that
\[\frac{\psi((x^0,y^0)+r(x,y))}{{r^\beta}}\to 0 \quad\text{as }r\searrow 0,\]
strongly in $L^2_{\textnormal{loc}}(\R^2)$ for each  $\beta \in
[0,N)$, and
\[\frac{\psi((x^0,y^0)+r(x,y))}{\sqrt{r^{-1}\int_{\partial B_r((x^0,y^0))}
\psi^2 \dhone}} \to
\frac{\rho^{N}|\sin(N\theta)|}{\sqrt{\int_0^{2\pi}\sin^2(N\theta)
d\theta}} \quad\text{as } r\searrow 0,\] strongly in
$W^{1,2}_{\textnormal{loc}}(B_1\setminus\{0\})$ and weakly in
$W^{1,2}(B_1)$, where $(x,y)=(\rho\cos\theta, \rho\sin\theta)$.
\end{theoremb}

\begin{figure}
\begin{center}
\input{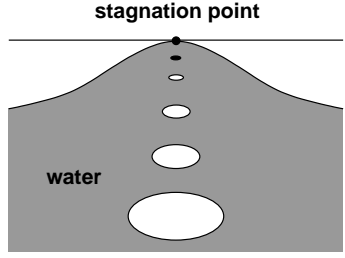}
\end{center}
\caption{A degenerate point}\label{fig1}
\end{figure}
Although the new dynamics suggested by Theorem B at degenerate
points cannot happen in the case of a finite number of air
components, there seems to be no obvious reason precluding the
scenario in Figure \ref{fig1} with an infinite number of air
components, and the situation is even less clear in the case of
inhomogeneous Neumann boundary conditions. Note that multiple air
components without surface tension have previously been considered
in \cite{gara}. It is noteworthy that while the water wave problem
has a variational structure, the solutions of interest are {\em not}
minimisers of the energy functional. Consequently, standard methods
in free-boundary problems based on non-degeneracy, which would in
the present case be the estimate \[\int_{\partial B_1((0,0))}
\frac{\psi((x^0,y^0)+r(x,y))}{{r^{3/2}}}\dhone \ge c_1
>0\quad\text{for all }r\in (0, r_0),\] do not apply.

As far as the water-wave problem is concerned, the new perspective
of our approach is that we work with the original variables $(\tilde
u,\tilde v)$ and use geometric methods as for example a {\em blow-up
analysis} in order to show that the scaled solution is close to a
homogeneous function. This part of the blow-up analysis works in $n$
dimensions and does not require {\em ad hoc} methods previously
applied to classify global solutions (see for example \cite{EV2}).
This also means that we do not require isolated singularities,
symmetry or monotonicity which had been assumed in all previous
results. Original tools in the present paper include the new {\em
Frequency Formula} (Theorem \ref{freq2}) which allows a  blow-up
analysis at degenerate points, where the scaling of the solution is
different from the invariant scaling of the equation, and leads 
in combination with the result \cite{evansmueller} by Evans and M\"uller
to
concentration compactness (Theorem \ref{deg2d}).

Large parts of the paper are written down for the non-physical but
mathematically interesting free-boundary problem in $n$ dimensions;
see for example the partial regularity result Proposition
\ref{partial} showing that non-degenerate stagnation points form a
set of dimension less or equal than $n-2$.

Our methods can still be applied when dropping the condition of
irrotationality of the flow (see \cite{EV2}, the forthcoming papers
\cite{vort} and \cite{vort2}, and \cite{CS,CS2} for a background on
water waves with vorticity).
Part of the methods extend even to water waves with surface tension
(see the forthcoming paper \cite{surf}).

It is interesting to observe that in his formal proof of the
conjecture, Stokes worked with the original variables $(\tilde
u,\tilde v)$ and approximated the velocity potential (the harmonic
conjugate of $-\psi$) by a {\em homogeneous function}. This is very
close in spirit to what we do on a rigorous level in the
Monotonicity Formula (Theorem \ref{elmon2}) and the Frequency
Formula (Theorem \ref{freq2}), so let us close our introduction with
a quotation taken from \cite[p. 226--227]{stokespaper}:

\begin{quotation}
Reduce the wave motion to steady motion by superposing a velocity
equal and opposite to that of propagation. Then a particle at the
surface may be thought of as gliding along a fixed smooth curve:
this follows directly from physical considerations, or from the
ordinary equation of steady motion. On arriving at a crest the
particle must be momentarily at rest, and on passing it must be
ultimately in the condition of a particle starting from rest down an
inclined or vertical plane. Hence the velocity must vary ultimately
as the square root of the distance from the crest.

Hitherto the motion has been rotational or not, let us now confine
ourselves to the case of irrotational motion. Place the origin at
the crest, refer the function $\phi$ to polar coordinates
$r,\theta$; $\theta$ being measured from the vertical, and consider
the value of $\phi$ very near the origin, where $\phi$ may be
supposed to vanish, as the arbitrary constant may be omitted. In
general $\phi$ will be of the form $\sum A_n r^n\sin n\theta +\sum
B_n r^n\cos n\theta$. In the present case $\phi$ must contain sines
only on account of the symmetry of the motion, as already shown (p.\
212), so that retaining only the most important term we may take
$\phi=Ar^n \sin n\theta$. Now for a point in the section of the
profile we must have $d\phi/d\theta=0$, and $d\phi/d\theta$ varying
ultimately as $r^{1/2}$. This requires $n=3/2$, and for the profile
that $3\theta /2 = \pi/2$, so that the two branches are inclined at
angles of $\pm 60^\circ$ to the vertical, and at an angle of
$120^\circ$ to each other, not of $90^\circ$ as supposed by Rankine.
\end{quotation}

\begin{acknowledgment}
We are very grateful to Stefan M\"uller, Pavel Plotnikov, John
Toland and Yoshihiro Tonegawa for helpful suggestions and
discussions.
\end{acknowledgment}

\section{Notation}
\label{notation}  We denote by $\chi_A$ the characteristic function
of the set $A$, and by $A\bigtriangleup B$ the set $(A\setminus
B)\cup(B\setminus A)$. For any real number $a$, the notation $a^+$
stands for $\max(a, 0)$. We denote by $x\cdot y$ the Euclidean inner
product in $\R^n \times \R^n$, by $\vert x \vert$ the Euclidean norm
in $\R^n$ and by $B_r(x^0):= \{x \in \R^n : \vert x-x^0 \vert < r\}$
the ball of center $x^0$ and radius $r$. We will use the notation
$B_r$ for $B_r(0)$, and denote by $\omega_n$ the $n$-dimensional
volume of $B_1$. Also, ${\mathcal L}^n$ shall denote the
$n$-dimensional Lebesgue measure and ${\mathcal H}^s$ the
$s$-dimensional Hausdorff measure. By $\nu$ we will always refer to
the outer normal on a given surface. We will use functions of
bounded variation $BV(U)$, i.e. functions $f\in L^1(U)$ for which
the distributional derivative is a vector-valued Radon measure. Here
$|\nabla f|$ denotes the total variation measure (cf. \cite{giu}).
Note that for a smooth open set $E\subset \R^n$, $|\nabla \chi_E|$
coincides with the surface measure on $\partial E$.
Last, we will use the notation $r \searrow 0$ in the meaning
of $r\to 0+$ and $r \nearrow 0$ in the meaning
of $r\to 0-$. 

\section{Notion of solution and monotonicity formula}
 Throughout the rest of the paper we work with a $n$-dimensional
generalization of the problem described in the Introduction. Let
$\Omega$ be a bounded domain in $\R^n$ which has a non-empty
intersection with the hyperplane $\{x_n=0\}$, in which to consider
the combined problem for fluid and air. We study solutions $u$, in a
sense to be specified, of the problem
\begin{align} \label{strongp}
 \Delta u &= 0 \quad\text{in } \Omega \cap \{u>0\},\\
{\vert \nabla u\vert}^2
  &=   x_n  \quad\text{on } \Omega\cap\partial
\{u>0\}.\non\end{align} (Note that, compared to the Introduction, we
have switched notation from $\psi$ to $u$ and we have ``reflected''
the problem at the hyperplane $\{x_n=0\}$.) Since our results are
completely local, we do not specify boundary conditions on
$\partial\Omega$.

We begin by introducing our notion of a {\em variational solution}
of problem (\ref{strongp}).

\begin{definition}[Variational Solution]
\label{vardef} We define $u\in W^{1,2}_{\textnormal{loc}}(\Omega)$
to be a {\em variational solution} of (\ref{strongp}) if $u\in
C^0(\Omega)\cap C^2(\Omega \cap \{ u>0\})$, $u\geq 0$ in $\Omega$
and $u\equiv 0$ in $\Omega\cap\{x_n\leq 0\}$,
 and the first variation with respect to
domain variations of the functional
\[ J(v) :=  \int_\Omega \left( {\vert \nabla v \vert}^2   +
x_n   \chi_{\{ v>0\}} \right)\] vanishes at $v=u  ,$ i.e.
\begin{align} 0 &= -\frac{d}{d\epsilon} J(u(x+\epsilon
\phi(x)))|_{\epsilon=0}  \non\\&= \int_\Omega \Big( {\vert \nabla u
\vert}^2 \div\phi-2 \nabla u D\phi\nabla u + x_n \chi_{\{
u>0\}}\div\phi + \chi_{\{ u>0\}} \phi_n\Big)\non\end{align} for any
$\phi \in C^1_0(\Omega;\R^n).$
\end{definition}

The assumption $u\in C^0(\Omega) \cap C^2(\Omega\cap \{u>0\})$ is
necessary in that it cannot be deduced from the other assumptions in
Definition \ref{vardef} by regularity theory, but it is rather mild
in the sense that it can be verified without effort for
``reasonable'' solutions, for example solutions obtained by a
diffuse interface approximation. Also we like to emphasise that
regularity properties of the free boundary, like for example finite
perimeter, are not required at all. Note for future reference that
the fact that $u$ continuous and nonnegative in $\Omega$, as well as
harmonic in $\{ u>0\}$, implies that $\Delta u$ is a nonnegative
Radon measure in $\Omega$ with support on $\Omega\cap\partial\{
u>0\}$.

We will also use {\em weak solutions} of (\ref{strongp}), i.e.
solutions in the sense of distributions. For a comparison of
variational and weak solutions see Lemma \ref{arevar}.

\begin{definition}[Weak Solution]\label{weak}
We define $u\in W^{1,2}_{\textnormal{loc}}(\Omega)$ to be a {\em
weak solution} of (\ref{strongp}) if the following are satisfied:
$u\in C^0(\Omega)$, $u\geq 0$ in $\Omega$ and $u\equiv 0$ in
$\Omega\cap\{x_n\leq 0\}$,
 $u$ is harmonic in $\{ u>0\} \cap \Omega$ and, for
every $\tau>0$, the topological free boundary $\partial \{ u>0\}
\cap \Omega \cap \{ x_n>\tau\}$ can be locally decomposed into an
$n-1$-dimensional $C^{2,\alpha}$-surface, relatively open to
$\partial \{ u>0\}$ and denoted by $\partial_{\textnormal{red}} \{
u>0\}$, and a singular set of vanishing $\mathcal{H}^{n-1}$-measure;
for an open neighborhood $V$ of each point $x^0 \in \Omega \cap \{
x_n>\tau\}$ of $\partial_{\textnormal{red}} \{ u>0\}$, $u\in
C^1(V\cap \overline{\{ u>0\}})$ satisfies \[|\nabla u(x)|^2 =
x_n\quad\text{on }V\cap
\partial_{\textnormal{red}} \{ u>0\}.\]
\end{definition}

\begin{remark}
(i) By \cite[Theorem 8.4]{AC}, the weak solutions in \cite{AC} with
$Q(x)=x_n^+$ satisfy Definition \ref{weak}.

(ii) By \cite[Theorem 3.5]{EV3}, the weak solutions in \cite{ST,EV3}
satisfy Definition \ref{weak}.
\end{remark}

\begin{lemma}\label{arevar}
  Any weak
solution of {\rm (\ref{strongp})} such that  \[|\nabla u|^2\le
Cx_n^+\quad\text{locally in }\Omega,\]
 is a variational
solution of {\rm (\ref{strongp})}. Moreover, $ \chi_{\{ u>0\}}$ is
locally in $\{x_n>0\}$ a function of bounded variation, and the
total variation measure $|\nabla \chi_{\{ u>0\}}|$ satisfies
\[r^{1/2-n}\int_{B_r(y)} \sqrt{x_n} |\nabla \chi_{\{ u>0\}}|\leq C_0\]
for all $B_r(y)\subset\subset \Omega$ such that $y_n=0$.
\end{lemma}

The proof follows \cite[Theorem 5.1]{cpde} and will be done in the
Appendix.

A first tool in our analysis is an extension of the monotonicity
formula in \cite{jga}, \cite[Theorem 3.1]{cpde} to the boundary
case. The roots of those monotonicity formulas
are harmonic mappings (\cite{schoen}, \cite{price}) and blow-up (\cite{pacard}).

\begin{theorem}[Monotonicity Formula]\label{elmon2}
Let $u$ be a variational solution of {\rm (\ref{strongp})}, let
$x^0\in\Omega$, and let
$\delta:=\textnormal{dist}(x^0,\partial\Om)/2$.

(i) Interior Case $x^0_n\ge 0$: The function \begin{align}
\Phi^{\textnormal{int}}_{x^0,u}(r) := r^{-n} \int_{B_r(x^0)} &\left(
{\vert \nabla u \vert}^2 +  x_n   \chi_{\{ u>0\} } \right)\non\\ &-
r^{-n-1} \int_{\partial B_r(x^0)} u^2 \dh, \non\end{align} defined
in $(0,\delta)$, satisfies the formula
\begin{align} \Phi^{\textnormal{int}}_{x^0,u}(\sigma)  - \Phi^{\textnormal{int}}_{x^0,u}(\varrho)  =&
\int_\varrho^\sigma r^{-n} \int_{\partial B_r(x^0)} 2 \left(\nabla u
\cdot \nu -  \frac{u}{r}\right)^2   \dh   \,dr
\non\\&+\int_\varrho^\sigma r^{-n-1}\int_{B_r(x^0)}(x_n-x_n^0)
\chi_{\{ u>0\}}\,dx\,dr,\non\end{align} for any
$0<\varrho<\sigma<\delta$.
 The absolute value of the
second term of the right-hand side is estimated by $\sigma-\varrho$
and is therefore $O(\sigma)$.

(ii) Boundary Case $x^0_n=0$: The function
\begin{align} \Phi^{\textnormal{bound}}_{x^0,u}(r) := r^{-n-1} \int_{B_r(x^0)}&\left( {\vert \nabla u
\vert}^2 +  x_n   \chi_{\{ u>0\}} \right)\non
\\ &- \frac{3}{2}   r^{-n-2} \int_{\partial
B_r(x^0)} u^2    \dh,\non\end{align} defined in $(0,\delta)$,
satisfies the formula
\[ \Phi^{\textnormal{bound}}_{x^0,u}(\sigma)  -  \Phi^{\textnormal{bound}}_{x^0,u}(\varrho)  =
\int_\varrho^\sigma r^{-n-1} \int_{\partial B_r(x^0)} 2 \left(\nabla
u \cdot \nu - \frac{3}{2}   \frac{u}{r}\right)^2   \dh   \,dr,\] for
any $0<\varrho<\sigma<\delta$.
\end{theorem}

\begin{remark}
\label{elrem} Let us assume that $x^0=0.$ Then the integrand on the
right-hand side of the Monotonicity Formula is a scalar multiple of
$ (\nabla u(x) \cdot x - \frac{3}{2}u(x))^2$, and therefore vanishes
if and only if $u$ is a homogeneous function of degree $3/2$.
\end{remark}

\begin{proof}[Proof of Theorem \ref{elmon2}] We start with a general observation:
for any $u\in W^{1,2}_{\textnormal{loc}}(\Om)$ and $\alpha\in\R$,
the following identity holds a.e. on $(0,\delta)$, where
$w_r(x)=u(x^0+rx)$:
\begin{align}
& \frac{d}{dr} \left(r^\alpha \int_{\partial B_r(x^0)} u^2
\dh\right)
 = \frac{d}{{dr}} \left(r^{\alpha+n-1}
\int_{\partial B_1} w_r^2    \dh\right)\label{easy}\\ \non &=
(\alpha+n-1) r^{\alpha-1} \int_{\partial B_r(x^0)} u^2    \dh
+ r^{\alpha+n-1} \int_{\partial B_1} 2w_r\nabla u(x^0+rx)\cdot x   \dh\\
&= (\alpha+n-1) r^{\alpha-1} \int_{\partial B_r(x^0)} u^2    \dh +
r^{\alpha} \int_{\partial B_r(x^0)} 2u\nabla u\cdot \nu    \dh.\non
\end{align}

Suppose now that $u$ is a variational solution of (\ref{strongp}).
For small positive $\kappa$ and $\eta_\kappa(t) :=
\max(0,\min(1,\frac{r- t}{ \kappa}))$, we take after approximation
$\phi_\kappa(x) := \eta_\kappa(|x-x^0|)(x-x^0) $ as a test function
in the definition of a variational solution. We obtain
\begin{align}\non 0 =& \int_{\Om}
\left( |\nabla u |^2+x_n\chi_{\{ u>0\}}\right)\left(n\eta_\kappa(|x-x^0|)+\eta_\kappa'(|x-x^0|)|x-x^0|\right)\\
&-2\int_{\Om}  |\nabla u |^2\eta_\kappa(|x-x^0|)+\nabla u\cdot
\frac{x-x^0}{|x-x^0|}\nabla u\cdot \frac{x-x^0}{|x-x^0|}
\eta'(|x-x^0|)|x-x^0|\non\\&
+\int_{\Om}\eta_\kappa(|x-x^0|)(x_n-x^0_n) \chi_{\{
u>0\}}.\non\end{align} Passing to the limit as $\kappa\to 0$, we
obtain, for a.e. $r \in (0,\delta)$,
\begin{align} 0=& n\int_{B_r(x^0)}\Big({\vert \nabla u \vert}^2 +
x_n  \chi_{\{ u>0\}} \Big)
\label{mff2}\\
&- r \int_{\partial B_r(x^0)}   \Big({\vert \nabla u \vert}^2 +
 x_n  \chi_{\{ u>0\}} \Big)\dh\non\\ &+ 2r\int_{\partial B_r (x^0)}
(\nabla u \cdot \nu)^2\,\dh - 2\int_{B_r(x^0)} {\vert \nabla u
\vert}^2\dh\non\\&+\int_{B_r(x^0)}(x_n-x^0_n)\chi_{\{ u>0\}}.\non
\end{align}
Observe that letting $\epsilon\to 0$ in \[\int_{B_r(x^0)} \nabla u
\cdot \nabla \max(u-\epsilon,0)^{1+\epsilon} =\int_{\partial
B_r(x^0)} \max(u-\epsilon,0)^{1+\epsilon} \nabla u \cdot \nu\, \dh\]
for a.e. $r \in (0,\delta)$, we obtain the integration by parts
formula
\begin{equation}\label{part2}
 \int_{B_r(x^0)} {\vert \nabla u \vert}^2 =\int_{\partial B_r(x^0)}
u \nabla u \cdot \nu\,  \dh \end{equation} for a.e. $r \in
(0,\delta).$

Now let for all $r\in (0,\delta)$,
\begin{align}
 U_{\textnormal{int}}(r)&:= r^{-n}
\int_{B_r(x^0)} \left( {\vert \nabla u \vert}^2
 + x_n\chi_{\{ u>0\}}\right),\non\\W_{\textnormal{int}}(r)&:=  r^{-n-1}
\int_{\partial B_r(x^0)} u^2\dh,\non\end{align} so that
$\Phi^{\textnormal{int}}_{x^0,u}=U_{\textnormal{int}}-W_{\textnormal{int}}$.
Note that, for a.e. $r\in (0,\delta)$,
\begin{align}U_{\textnormal{int}}'(r)&= -nr^{-n-1}\int_{B_r(x^0)}\Big({\vert \nabla u \vert}^2
 + x_n  \chi_{\{ u>0\}} \Big)
\non\\
&\qquad+ r^{-n} \int_{\partial B_r(x^0)} \Big({\vert \nabla u
\vert}^2 +
 x_n  \chi_{\{ u>0\}}     \Big)\dh .\non\end{align}
 It follows, using (\ref{mff2}) and (\ref{part2}), that for a.e. $r\in
(0,\delta)$,
\begin{align}\label{duut2} U_{\textnormal{int}}'(r)&= 2
r^{-n}\int_{\partial B_r(x^0)} (\nabla u \cdot \nu)^2\,
\dh-2r^{-n-1}\int_{\partial B_r(x^0)} u \nabla u \cdot \nu
 \dh\\&\qquad+r^{-n-1}\int_{B_r(x^0)}(x_n-x^0_n)\chi_{\{ u>0\}}.\non\end{align}
On the other hand, plugging $\alpha:= -n-1$ into (\ref{easy}), we
obtain that for a.e. $r\in (0,\delta)$,
 \be\label{dwt2}
W_{\textnormal{int}}'(r)=2r^{-n-1}\int_{\partial B_r(x^0)} u \nabla
u\cdot\nu\dh-2r^{-n-2}\int_{\partial B_r(x^0)} u^2\dh.\ee Combining
(\ref{duut2}) and (\ref{dwt2}) yields (i).

Next, let for all $r\in (0,\delta)$, \begin{align}
U_{\textnormal{bound}}(r)&:= r^{-n-1} \int_{B_r(x^0)} \left( {\vert
\nabla u \vert}^2
 + x_n\chi_{\{ u>0\}}\right),\non\\ W_{\textnormal{bound}}(r)&:=  r^{-n-2}   \int_{\partial
B_r(x^0)} u^2\dh,\non
\end{align}
so that $\Phi^{\textnormal{bound}}_{x^0,u}=
U_{\textnormal{bound}}-3/2\, W_{\textnormal{bound}}$. Now observe
that, in the case when $x_n^0=0$, formula (\ref{mff2}) means that
\begin{align} 0= (n+1)&\int_{B_r(x^0)}\Big({\vert \nabla u \vert}^2 +
x_n  \chi_{\{ u>0\}} \Big)
\label{mff3}\\
&- r \int_{\partial B_r(x^0)}\Big({\vert \nabla u \vert}^2 +
 x_n  \chi_{\{ u>0\}} \Big)\dh\non\\ &+ 2r\int_{\partial B_r (x^0)} (\nabla
u \cdot \nu)^2\,\dh - 3\int_{B_r(x^0)} {\vert \nabla u \vert}^2.\non
\end{align}

Also, for a.e. $r\in (0,\delta)$,
\begin{align} U_{\textnormal{bound}}'(r)= -(n+1)r^{-n-2}\int_{B_r(x^0)}\Big({\vert \nabla u \vert}^2
 + x_n  \chi_{\{ u>0\}} \Big)
\non\\
+ r^{-n-1} \int_{\partial B_r(x^0)}   \Big({\vert \nabla u \vert}^2
+  x_n  \chi_{\{ u>0\}}\Big)\non\end{align} It follows, using
(\ref{mff3}) and (\ref{part2}), that for a.e. $r\in (0,\delta)$,
\begin{align}\label{duu2} U_{\textnormal{bound}}'(r)=& 2
r^{-n-1}\int_{\partial B_r(x^0)} (\nabla u \cdot \nu)^2\,
\dh\\&-3r^{-n-2}\int_{\partial B_r(x^0)} u \nabla u \cdot \nu
 \dh.\non\end{align} On the other hand,
plugging $\alpha := -n-2$ into (\ref{easy}), we obtain that for a.e.
$r\in (0,\delta)$, \be\label{dw2}
W_{\textnormal{bound}}'(r)=2r^{-n-2}\int_{\partial B_r(x^0)} u
\nabla u\cdot\nu\dh-3r^{-n-3}\int_{\partial B_r(x^0)} u^2\dh.\ee
Combining (\ref{duu2}) and (\ref{dw2}) yields (ii).
\end{proof}

\section{Densities}
\label{propdens1} From Theorem \ref{elmon2} we infer that the
functions $\Phi^{\textnormal{int}}_{x^0,u}$ and
$\Phi^{\textnormal{bound}}_{x^0,u}$ have right limits
\[\Phi^{\textnormal{int}}_{x^0,u}(0+)= \lim_{r\searrow 0}
\Phi^{\textnormal{int}}_{x^0,u}(r)\in [-\infty,\infty),\quad
\Phi^{\textnormal{bound}}_{x^0,u}(0+)= \lim_{r\searrow 0}
\Phi^{\textnormal{bound}}_{x^0,u}(r)\in [-\infty,\infty).\]
 In this
section we derive structural properties of these ``densities''
\[\Phi^{\textnormal{int}}_{x^0,u}(0+),\Phi^{\textnormal{bound}}_{x^0,u}(0+).\]
The term ``density'' is justified somewhat by Lemma \ref{density_2}
(i), (ii).

Note that most of the statements concerning
$\Phi^{\textnormal{int}}_{x^0,u}$ will not be used in subsequent
sections but serve to illustrate differences between the boundary
and interior case.

\begin{lemma}\label{density_1}
Let $u$ be a variational solution of {\rm (\ref{strongp})} and
suppose that
\[|\nabla u|^2 \le Cx_n^+\quad\text{locally in }\Omega.\]
 Then:

(i) Let $x^0\in \Om$ be such that $x^0_n>0$. Then
$\Phi^{\textnormal{int}}_{x^0,u}(0+)$ is finite if $u(x^0)=0$, and
$\Phi^{\textnormal{int}}_{x^0,u}(0+)=-\infty$ otherwise.

(ii)  Let $x^0\in \Om$ be such that $x^0_n=0$. Then
$\Phi^{\textnormal{bound}}_{x^0,u}(0+)$ is finite. (Note that $u=0$
in $\{ x_n=0\}$ by assumption.)

(iii)  Let $x^0\in \Om$ be such that $x^0_n>0$ and
 $u(x^0)=0$, and let $0<r_m\searrow 0$ as $m\to
\infty$ be a sequence such that the {\em blow-up} sequence
\[u_m(x) := {u(x^0+{r_m}x)/ {r_m}}\] converges weakly in
 $W^{1,2}_{\textnormal{loc}}(\R^{n})$ to
a blow-up limit $u_0$. Then $u_0$ is a homogeneous function of
degree $1$, i.e. $u_0(\lambda x)  = \lambda u_0(x)$.

(iv)  Let $x^0\in \Om$ be such that $x^0_n=0$, and let
$0<r_m\searrow 0$ as $m\to \infty$ be a sequence such that the {\em
blow-up} sequence
\[u_m(x) := {u(x^0+{r_m}x)/ {r_m^{3/2}}}\] converges weakly
in $W^{1,2}_{\textnormal{loc}}(\R^{n})$
 to a blow-up limit $u_0$. Then $u_0$ is a homogeneous function of
degree $3/2$, i.e. $u_0(\lambda x)  = \lambda^{3/2} u_0(x)$.

(v) Let $u_m$ be a converging sequence of (iii) or (iv). Then $u_m$
converges {\em strongly} in $W^{1,2}_{\textnormal{loc}}(\R^n)$.
\end{lemma}

\begin{proof} (i),(ii): If $u(x^0)=0$, the finiteness claims follow directly
from the growth assumption $|\nabla u|^2 \le Cx_n^+$. If $x^0_n>0$
and $u(x^0)>0$, then, since $\vert\nabla u\vert^2\le Cx_n^+$ by
assumption, we obtain that $\Phi^{\textnormal{int}}_{x^0,u}(r) \le
C_1 - C_2 r^{-2}$ for $r \le r_0 ,$ implying that
$\Phi^{\textnormal{int}}_{x^0,u}(0+) = -\infty .$

(iii),(iv): For each $0<\sigma<\infty$ the sequence $u_m$ is by
assumption bounded in $C^{0,1}(B_\sigma)) .$ From the Monotonicity
Formula (Theorem \ref{elmon2}) we infer therefore, setting
$\alpha=1$ in the interior case and $\alpha=3/2$ in the boundary
case, that for all $0<\varrho<\sigma<\infty$,
\[ \int_\varrho^\sigma
 \int_{\partial B_r}\left(\nabla u_m(x) \cdot x -  \alpha
u_m(x)\right)^2\dh\,dr\to 0\quad\text{as }m\to\infty,
\] which yields the desired homogeneity of $u_0 .$

(v): The proof follows \cite[Lemma 7.2]{cava}. In order to show
strong convergence of $u_m$ in $W^{1,2}_{\textnormal{loc}}(\R^n)$,
it is sufficient, in view of the weak $L^2$-convergence of $\nabla
u_m$, to show that \[\limsup_{m\to \infty}\int_{\R^n} |\nabla u_m|^2
\eta\le \int_{\R^n} |\nabla u_0|^2 \eta\] for each $\eta\in
C^1_0(\R^n)$. Using the uniform convergence, the continuity of
$u_0$, as well as the fact that $u_0$ is harmonic in $\{ u_0>0\}$,
we obtain as in the proof of (\ref{part2}) that \[\int_{\R^n}
|\nabla u_m|^2\eta =-\int_{\R^n} u_m \nabla u_m\cdot \nabla \eta \to
-\int_{\R^n} u_0 \nabla u_0\cdot \nabla \eta =\int_{\R^n} |\nabla
u_0|^2\eta\] as $m\to \infty$. It follows that $u_m$ converges to
$u_0$ strongly in $W^{1,2}_{\textnormal{loc}}(\R^n)$ as
$m\to\infty$. \end{proof}

\begin{lemma}\label{density_2}
Let $u$ be a variational solution of {\rm (\ref{strongp})} and
suppose that
\[|\nabla u|^2 \le Cx_n^+\quad\text{locally in }\Omega.\] Then:

(i)  Let $x^0\in \Om$ be such that $x^0_n>0$ and
 $u(x^0)=0$. Then
\[\Phi^{\textnormal{int}}_{x^0,u}(0+)  =  x^0_n\lim_{r\searrow 0} r^{-n}
\int_{B_r(x^0)}\chi_{\{ u>0\}},\] and in particular
$\Phi^{\textnormal{int}}_{x^0,u}(0+) \in [0,+\infty)$.
 Moreover,
 $\Phi^{\textnormal{int}}_{x^0,u}(0+)=0$ implies that $u_0=0$ in $\R^n$ for
each blow-up limit $u_0$ of Lemma \ref{density_1} (iii).

(ii)  Let $x^0\in \Om$ be such that $x^0_n=0$. Then
 \[\Phi^{\textnormal{bound}}_{x^0,u}(0+)  =
\lim_{r\searrow 0} r^{-n-1} \int_{B_r(x^0)} x_n^+\chi_{\{ u>0\}},\]
and in particular $\Phi^{\textnormal{bound}}_{x^0,u}(0+)\in
[0,+\infty)$. Moreover, $\Phi^{\textnormal{bound}}_{x^0,u}(0+)=0$
implies that $u_0=0$ in $\R^n$ for each blow-up limit $u_0$ of Lemma
\ref{density_1} (iv).

(iii) The function $x\mapsto\Phi^{\textnormal{int}}_{x,u}(0+)$ is
upper semicontinuous in  $\{ x_n>0\}$.

(iv) The function $x\mapsto \Phi^{\textnormal{bound}}_{x,u}(0+)$ is
upper semicontinuous in $\{ x_n=0\}$.

(v) Let $u_m$ be a sequence of variational solutions of
{\rm(\ref{strongp})} which converges strongly to $u_0$ in
$W^{1,2}_{\textnormal{loc}}(\R^n)$ and such that $\chi_{\{u_m>0\}}$
converges weakly in $L^2_{\textnormal{loc}}(\R^n)$ to $\chi_0$. Then
$u_0$ is a variational solution of {\rm (\ref{strongp})} and
satisfies the Monotonicity Formula, but with $\chi_{\{u_0>0\}}$
replaced by $\chi_0$. Moreover, for each $x^0\in \Om$,
and all instances of $\chi_{\{u_0>0\}}$
replaced by $\chi_0$,
\[\Phi^{\textnormal{int}}_{x^0,u_0}(0+)\geq\limsup_{m\to\infty}\Phi^{\textnormal{int}}_{x^0,u_m}(0+)\]
in the interior case $x^0_n>0$, and
\[\Phi^{\textnormal{int}}_{x^0,u_0}(0+)\geq\limsup_{m\to\infty}\Phi^{\textnormal{int}}_{x^0,u_m}(0+)\]
in the boundary case $x^0_n=0$.
\end{lemma}

\begin{proof} (i),(ii): Let us take a sequence $r_m\searrow 0$ such that $u_m$
defined in Lemma \ref{density_1} (iii), (iv) converges weakly in
$W^{1,2}_{\textnormal{loc}}(\R^n)$ to a function $u_0 .$ Using Lemma
\ref{density_1} (v) and the homogeneity of $u_0$,  in the interior
case we obtain that
\begin{align} \lim_{m\to \infty} \Phi^{\textnormal{int}}_{x^0,u}(r_m)& =
\int_{B_1} |\nabla u_0|^2 - \int_{\partial B_1} u_0^2 \dh +
x^0_n\lim_{r\searrow 0} r^{-n}
\int_{B_r(x^0)}\chi_{\{ u>0\}}\non\\
 &= x^0_n\lim_{r\searrow 0} r^{-n} \int_{B_r(x^0)}\chi_{\{ u>0\}},\non\end{align}
(the limit here exists because $\lim_{r\searrow 0}
\Phi^{\textnormal{int}}_{x^0,u}(r)$ exists), while in the boundary
case we obtain that
\begin{align} \lim_{m\to \infty} \Phi^{\textnormal{bound}}_{x^0,u}(r_m)
 &= \int_{B_1} |\nabla u_0|^2 - \frac{3}{2}\int_{\partial B_1} u_0^2
\dh + \lim_{r\searrow 0} r^{-n-1} \int_{B_r(x^0)} x_n^+\chi_{\{
u>0\}}\non\\
& =  \lim_{r\searrow 0} r^{-n-1} \int_{B_r(x^0)} x_n^+\chi_{\{
u>0\}}\non.\end{align}
  Thus $\Phi^{\textnormal{int}}_{x^0,u}(0+)\ge 0$ in the
interior case, $\Phi^{\textnormal{bound}}_{x^0,u}(0+)\ge 0$ in the
boundary case, and equality in either case implies that for each
$\tau>0$, $u_m$ converges to $0$ in measure in the set $\{x_n
>\tau\}$ as $ m \to \infty $, and consequently $u_0=0$ in $\R^n$.

(iii),(iv): For each $\delta>0$ and $K<+\infty$ we obtain from the
Monotonicity Formula (Theorem \ref{elmon2}) that in the interior
case
\begin{align} \Phi^{\textnormal{int}}_{x,u}(0+) & \le  \Phi^{\textnormal{int}}_{x,u}(r)  \le
 \Phi^{\textnormal{int}}_{x^0,u}(r)+\frac{\delta}{2}\non\\ &\le \left\{\begin{array}{ll}
\Phi^{\textnormal{int}}_{x^0,u}(0+)+\delta&, \text{ if } \Phi^{\textnormal{int}}_{x^0,u}(0+)>-\infty,\\
-K &, \text{ if }
\Phi^{\textnormal{int}}_{x^0,u}(0+)=-\infty,\end{array}\right.\non\end{align}
 and in
the boundary case
\[ \Phi^{\textnormal{bound}}_{x,u}(0+)  \le  \Phi^{\textnormal{bound}}_{x,u}(r)  \le
 \Phi^{\textnormal{bound}}_{x^0,u}(r)+ \frac{\delta}{2}\leq
 \Phi^{\textnormal{bound}}_{x^0,u}(0+)+\delta,\] if we choose for fixed
$x^0$ first $r>0$ and then $\vert x-x^0\vert$ small enough.
\end{proof}

(v) The fact that $u_0$ is a variational solution of (\ref{strongp})
and satisfies the Monotonicity Formula in the sense indicated
follows directly from the convergence assumption. The proof for the
rest of the claim follows by the same argument as in (iii), (iv).

\begin{lemma}\label{density_3}
Let $u$ be a variational solution of {\rm (\ref{strongp})} and
suppose that
\[|\nabla u|^2 \le Cx_n^+\quad\text{locally in }\Omega.\]
 Then
$\Phi^{\textnormal{int}}_{x^0,u}(0+)=0$ implies that $u\equiv 0$ in
some open $n$-dimensional ball containing $x^0$.
\end{lemma}

\begin{proof} By the upper semicontinuity Lemma \ref{density_2} (iii), $\Phi^{\textnormal{int}}_{x,u}(0+)\le
\epsilon$ in $B_\delta(x^0)\subset \Omega$ for some
$\delta\in(0,x^0_n)$. Suppose towards a contradiction that
$u\not\equiv 0$ in $B_\delta(x^0)$. Then there exist a ball
$A\subset \{ u>0\}\cap B_\delta(x^0)$ and $z\in \partial A \cap \{
u=0\}$. It follows that \[\Phi^{\textnormal{int}}_{z,u}(0+) =  z_n
\lim_{r\searrow 0}r^{-n} \int_{B_r(z)} \chi_{\{ u>0\}} \ge  z_n
\frac{\omega_n}{2},\] a contradiction for sufficiently small
$\epsilon$.\end{proof}

Unfortunately, a boundary version of Lemma \ref{density_3}, stating
that boundary density $0$ at $x^0$ implies the solution being $0$ in
an open $n$-dimensional ball with center $x^0$, cannot be obtained
in the same way. Instead we prove in the two-dimensional case the
following result.

\begin{lemma}\label{zero}
Let $n=2$, let $u$ be a weak solution of {\rm(\ref{strongp})} and
suppose that
\[|\nabla u|^2 \leq x_2^+\quad\text{in }\Omega.\]
 Then $\Phi^{\textnormal{bound}}_{x^0,u}(0+)=0$ implies that $u\equiv 0$ in some
open $2$-dimensional ball containing $x^0$.
\end{lemma}

\begin{proof} Suppose towards a contradiction that $x^0\in \partial\{
u>0\}$, and let us take a blow-up sequence \[u_m(x) :=
{u(x^0+{r_m}x)/{r_m^{3/2}}}\] converging weakly in
$W^{1,2}_{\textnormal{loc}}(\R^n)$ to a blow-up limit $u_0$. Lemma
\ref{density_2} (iv) shows that $u_0=0$ in $\R^2$. Consequently,
\begin{equation}\label{meas0}
0\gets  \Delta u_m (B_2)\ge \int_{B_2\cap
\partial_{\textnormal{red}} \{ u_m>0\}} \sqrt{x_2} \,d\mathcal{H}^1
\quad\text{as }m\to\infty.\end{equation} (Recall that $\Delta u$ is
a nonnegative Radon measure in $\Omega$.) On the other hand, there
is at least one connected component $V_m$ of $\{ u_m>0\}$ touching
the origin and containing by the maximum principle a point $x^m\in
\partial A$, where $A = (-1,1)\times (0,1)$. If $\max\{x_2: x \in
V_m\cap \partial A\}\not\to 0$ as $m\to\infty$, we immediately
obtain a contradiction to (\ref{meas0}). If $\max\{ x_2: x \in
V_m\cap \partial A\}\to 0$, we use the free-boundary condition as
well as $|\nabla u|^2 \leq x_2^+$ to obtain \[0= \Delta u_m(V_m\cap
A) \le \int_{V_m\cap \partial A} \sqrt{x_2} \,d\mathcal{H}^1 -
\int_{A\cap
\partial_{\textnormal{red}} V_m} \sqrt{x_2} \,d\mathcal{H}^1.\] However
$\int_{V_m\cap \partial A} \sqrt{x_2} \,d\mathcal{H}^1$ is the
unique minimiser of $\int_{\partial D} \sqrt{x_2} \,d\mathcal{H}^1$
with respect to all open sets $D$ with $D=V_m$ on $\partial A$. So
$V_m$ cannot touch the origin, a contradiction.\end{proof}

\begin{remark}
Note that we have not really used the full information contained in
the weak formulation. What we have used is the {\em inequality}
$\Delta u \ge \sqrt{x_2} \mathcal{H}\lfloor
\partial_{\textnormal{red}}\{ u>0\}$ (which is true for any limit of
the singular perturbation considered in \cite{calc}) and the fact
that we can locate a non-empty portion of
$\partial_{\textnormal{red}}\{ u>0\}$ touching $x^0$.
\end{remark}

In higher dimensions it is not so clear whether {\em cusps} can be
excluded. Of course that does not happen for Lipschitz free
boundaries:
\begin{lemma}\label{zero2}
Let $u$ be a variational solution of {\rm (\ref{strongp})} and
suppose that
\[|\nabla u|^2 \le Cx_n^+\quad\text{locally in }\Omega,\] and
that $\{ u>0\}$ is locally a Lipschitz set. Then
$\Phi^{\textnormal{bound}}_{x^0,u}(0+)=0$ implies that $u\equiv 0$
in some open $n$-dimensional ball containing $x^0$.
\end{lemma}

\begin{proof} This is an immediate consequence of Lemma \ref{density_2}
(ii) and the Lipschitz continuity.\end{proof}

\begin{proposition}[Two-dimensional Case]\label{2dim}
Let $n=2$, let $u$ be a variational solution of {\rm
(\ref{strongp})}, and suppose that
\[|\nabla u|^2 \le Cx_2^+\quad\text{locally in }\Omega.\] Let $x^0\in\Om$
be such that $u(x^0)=0$, and suppose that
\[r^{-1}\int_{B_r(x^0)} |\nabla \chi_{\{ u>0\}}|\le C_0\] for
all $r>0$ such that $B_r(x^0)\subset\subset \Omega$ in the interior
case, and that \[r^{-3/2}\int_{B_r(x^0)} \sqrt{x_2} |\nabla \chi_{\{
u>0\}}|\le C_0\] for all $r>0$ such that $B_r(x^0)\subset\subset
\Omega$ in the boundary case.

(i) Interior Case $x^0_2>0$: The only possible blow-up limits are
\[u_0(x)= \sqrt{x^0_2}\max(x\cdot e,0)\qquad\text{and}\qquad u_0(x)=\gamma |x\cdot e|,\]
where $e$ is a unit vector and $\gamma$ is a nonnegative constant.
In the case $u_0(x)= \sqrt{x^0_2}\max(x\cdot e,0)$ the corresponding
density value is $\omega_2/2$, in the case $u_0(x)=\gamma |x\cdot
e|$ with $\gamma>0$ the density is $\omega_2$, while in the case
$u_0=0$ the density may be either $0$ or $\omega_2$.

(ii) Boundary Case $x^0_2=0$: The only possible blow-up limits are
\[u_0(\rho,\theta)= \frac{\sqrt{2}}{3}\rho^{3/2}\max(\cos(3(\theta-\pi/2)/2),0),\]
 with corresponding density
 \[ \int_{B_1} x_2^+\chi_{\{
\cos(3(\theta-\pi/2)/2)>0\}},\] and $u_0(x)=0$, with possible values
of the density \[ \int_{B_1} x_2^+\quad\text{and}\quad 0.\]
\end{proposition}

\begin{proof} Consider a blow-up sequence $u_m$ as in
Lemma \ref{density_1}, where $r_m\searrow 0$, with blow-up limit
$u_0$. Because of the strong convergence of $u_m$ to $u_0$ in
$W^{1,2}_{\textnormal{loc}}(\R^2)$ and the compact embedding from
$BV$ into $L^1$, $u_0$ is a homogeneous solution of \be
\label{dmvint} 0 = \int_{\R^2} \Big( {\vert \nabla u_0 \vert}^2
\div\phi - 2 \nabla u_0 D\phi \nabla u_0 \Big) +x^0_2 \int_{\R^2}
\chi_0 \div\phi \ee for any $\phi \in C^1_0(\R^2;\R^2)$ in the
interior case, and of \be \label{dmvbound}0
 = \int_{\R^2} \Big( {\vert \nabla u_0 \vert}^2 \div\phi -
2 \nabla u_0 D\phi \nabla u_0 \Big)
 + \int_{\R^2} \Big(x_2 \chi_0 \div\phi + \chi_0 \phi_2\Big)
\ee for any $\phi \in C^1_0(\R^2;\R^2)$ in the boundary case, where
 $\chi_0$ is the strong $L^1_{\textnormal{loc}}$-limit of
$\chi_{\{u_m>0\}}$ along a subsequence. The values of the function
$\chi_0$ are almost everywhere in $\{0, 1\}$, and the locally
uniform convergence of $u_m$ to $u_0$ implies that $\chi_0=1$ in $\{
u_0>0\}$. The homogeneity of $u_0$ and its harmonicity in
$\{u_0>0\}$ show that each connected component of $\{u_0>0\}$ is a
half-plane passing trough the origin in the interior case, and a
cone with vertex at the origin and of opening angle $120^\circ$ in
the boundary case. Also, (\ref{dmvint}) and (\ref{dmvbound}) imply
 $\chi_0$ is  constant in the connected set $\{ u_0=0\}^\circ$.

Consider first the case when $\{u_0>0\}$ has exactly one connected
component. Let $z$ be an arbitrary  point in $
\partial\{ u_0=0\} \setminus \{ 0\}$. Note that the normal to
$\partial \{ u_0=0\}$ has the constant value
$\nu(z)$ in $B_\delta(z)$ for some $\delta>0$. Plugging in
$\phi(x):= \eta(x)\nu(z)$ into (\ref{dmvint}) and (\ref{dmvbound}),
where $\eta \in C^1_0(B_\delta(z))$ is arbitrary, and integrating by
parts, it follows that
\begin{equation}\label{dmvint2}0= \int_{\partial\{ u_0>0\}} \left(
-|\nabla u_0|^2 + x^0_2 (1-\bar\chi_0)\right)\eta
\,d\mathcal{H}^1\end{equation} in the interior case and that
\begin{equation}\label{dmvbound2}0= \int_{\partial\{ u_0>0\}} \left(
-|\nabla u_0|^2 + x_2 (1-\bar\chi_0)\right)\eta
\,d\mathcal{H}^1\end{equation} in the boundary case. Here
$\bar\chi_0$ denotes the constant value of $\chi_0$ in $\{
u_0=0\}^\circ$. Note that by Hopf's principle, $\nabla u_0\cdot
\nu\ne 0$ on $B_\delta(z)\cap \partial\{ u_0>0\}$. In both interior
and boundary case it follows therefore that $\bar\chi_0\neq 1$, and
hence necessarily $\bar\chi_0=0$. We deduce from (\ref{dmvint2}) and
(\ref{dmvbound2}) that $|\nabla u_0|^2=x^0_2$ on $\partial \{
u_0>0\}$ in the interior case and that $|\nabla u_0|^2 =x_2$ on
$\partial \{ u_0>0\}$ in the boundary case. Computing the ODE
solution $u_0$ on $\partial B_1$ yields the statement of the
Proposition in the case under consideration.

In the case $u_0= 0$, (\ref{dmvint}) and (\ref{dmvbound}) show that
$\chi_0$ is constant in $\R^2$. Its value may be either $0$ or $1$.

Last, consider the situation when, in the interior case, the set
$\{u_0>0\}$ has two connected components. The argument for
(\ref{dmvint2}) now yields that the constant values of $|\nabla
u_0|^2$ on either side of $\partial \{u_0>0\}$ are equal. This
yields the statement of the Proposition.

 \end{proof}

\section{Partial regularity of non-degenerate solutions}

\begin{definition}[Stagnation Points]
Let $u$ be a variational solution of {\rm (\ref{strongp})}. We call
$S^u := \{ x\in \Omega: x_n=0 \text{ and } x\in \partial\{ u>0\}\}$
the set of stagnation points.
\end{definition}

\begin{definition}[Non-degeneracy, Density Condition]\label{nondeg}
Let $u$ be a variational solution of {\rm (\ref{strongp})}.

(i) We say that a point $x^0\in \Omega\cap
\partial\{ u>0\}\cap\{ x_n=0\}$ satisfies property $(N)$  if
\[\liminf_{r\searrow 0}r^{-n-3} \int_{B_r(x^0)} u^2>0.\] Moreover we define
for each $\kappa>0$ and $\varsigma>0$ the set
\[N^u_{\varsigma,\kappa} := \{ x^0\in \Omega\cap \partial\{
u>0\}\cap\{ x_n=0\}: r^{-n-3} \int_{B_r(x^0)} u^2\ge \kappa
\quad\text{ for } r\in (0,\varsigma]\}.\]

(ii) We say that a point $x^0\in \Omega\cap \partial\{ u>0\}\cap\{
x_n=0\}$ satisfies property $(D)$  if
\begin{align}0&<\liminf_{r\searrow 0}r^{-n-1} \int_{B_r(x^0)} x_n^+
\chi_{\{ u>0\}}\non\\
&\le \limsup_{r\searrow 0}r^{-n-1} \int_{B_r(x^0)} x_n^+ \chi_{\{
u>0\}}\non\\ &< \int_{B_1} x_n^+.\non\end{align}
\end{definition}

Note that $\cup_{\varsigma,\kappa}N^u_{\varsigma,\kappa}$ is the set
of all points satisfying property (N).

\begin{lemma}\label{equiv}
Let $u$ be a variational solution of {\rm (\ref{strongp})} and
suppose that
\[|\nabla u|^2 \le Cx_n^+\quad\text{locally in }\Omega,\]
 and that
\[r^{1/2-n}\int_{B_r(y)} \sqrt{x_n} |\nabla \chi_{\{ u>0\}}|\le C_0\]
for all $B_r(y)\subset\subset\Omega$ such that $y_n=0$.
 Then properties (N) and (D) are equivalent.
\end{lemma}

\begin{proof} (D) $\Rightarrow$ (N): Consider a blow-up limit $u_0$ of the
sequence \[u_m(x) := u(x^0+r_m x)/r_m^{3/2},\] where $r_m\searrow
0$, and suppose towards a contradiction that $u_0=0$. Passing to the
limit in the domain variation equation we obtain
\begin{align}0&=\int_{\R^n} \Big( {\vert \nabla u_0 \vert}^2 \div
\phi   -   2 \nabla u_0 D\phi \nabla u_0 + x_n \chi_0 \div\phi
  + \chi_0\phi_n\Big)\non\\
&= \int_{\R^n} \Big( x_n \chi_0 \div\phi
  +  \chi_0\phi_n\Big)\non\end{align} for any $\phi \in
C^1_0(\R^n;\R^n)$, where $\chi_0$ is the limit of $\chi_{\{u_m>0\}}$
with respect to a subsequence. This implies that $\chi_0$ is a
constant function. On the other hand, the condition on $|\nabla
\chi_{\{ u>0\}}|$ implies that the values of $\chi_0$ are almost
everywhere in $\{0,1\}$, and then condition (D) shows that the
function $\chi_0$ is not constant, a contradiction.

(N) $\Rightarrow$ (D): The proof draws on \cite[Proof of Proposition
9.1]{calc}. Let us again consider a blow-up limit $u_0$ of the
sequence \[u_m(x):= u(x^0+r_m x)/r_m^{3/2},\] and suppose towards a
contradiction that $\chi_0 := \lim_{m\to \infty}
\chi_{\{u_m>0\}}\equiv 1$. By the Monotonicity Formula (which holds
for $u_0$ with $\chi_{\{ u_0>0\}}$ replaced by $\chi_0$) and the
growth estimate we obtain for each point $x$ such that $x_n=0$,
\begin{align}0&\gets \Phi^{\textnormal{bound}}_{x,u_0}(\sigma)  -
\Phi^{\textnormal{bound}}_{0,u_0}(\sigma)\non\\
&=\Phi^{\textnormal{bound}}_{x,u_0}(\sigma)
 -  \Phi^{\textnormal{bound}}_{0,u_0}(0+)\non\\&=\Phi^{\textnormal{bound}}_{x,u_0}(\sigma)
 -  \Phi^{\textnormal{bound}}_{x,u_0}(0+)\non\\
  &=
\int_0^\sigma r^{-n-1} \int_{\partial B_r(x)} 2 \left(\nabla u \cdot
\nu - \frac{3}{2}\frac{u}{r}\right)^2   \dh \,dr\non\end{align} as
$\sigma\to+\infty$. But this means that $u_0$ is homogeneous of
degree $3/2$ with respect to each point $x$ such that $x_n=0$. It
follows that $u_0$ depends only on the $x_n$-variable. Thus
$u_0(x)=\alpha (x_n^+)^{3/2}$ for some $\alpha \ge 0$, a
contradiction to the definition of variational solution unless
$\alpha=0$.
\end{proof}

\begin{proposition}[Two-dimensional Case]\label{isolated}
Let $n=2$, let $u$ be a variational solution of {\rm
(\ref{strongp})}, suppose that
\[|\nabla u|^2 \le Cx_2^+\quad\text{locally in }\Omega,\]
 and that
\[r^{-3/2}\int_{B_r(y)} \sqrt{x_2} |\nabla \chi_{\{ u>0\}}|\le C_0\]
for all $B_r(y)\subset\subset\Omega$ such that $y_n=0$. At each
non-degenerate stagnation point $x^0$, the density
$\Phi^{\textnormal{bound}}_{x^0,u}(0+)$ has the value
\[\int_{B_1} x_2^+\chi_{\{ \cos(3(\theta-\pi/2)/2)>0\}}\]
and
\[\frac{u(x^0+rx)}{r^{3/2}}\to \frac{\sqrt{2}}{3}
\rho^{3/2}\max(\cos(3(\theta-\pi/2)/2),0)\quad\text{as }r\searrow
0,\] strongly in $W^{1,2}_{\textnormal{loc}}(\R^2)$ and locally
uniformly on $\R^2$, where $x=(\rho\cos\theta,\rho\sin\theta)$.
Moreover,
\[\mathcal{L}^2(B_1\cap (\{x:
u(x^0+rx)>0\}\bigtriangleup\{\cos(3(\theta-\pi/2)/2)>0\}))\to
0\quad\text{as }r\searrow 0,\] and, for each $\delta>0$,
\[r^{-3/2}\Delta
u((x^0+B_r)\cap(\{\cos(3(\theta-\pi/2)/2)>\delta\}\cup
\{\cos(3(\theta-\pi/2)/2)<-\delta\}))\to 0\] as $r\searrow 0$.
(Recall that $\Delta u$ is a nonnegative Radon measure in $\Omega$.)

\end{proposition}

\begin{proof} The value of the density and the uniqueness
of the blow-up limit follow directly from Proposition \ref{2dim}
(ii) and the non-degeneracy assumption.

Let $r_m\searrow 0$ be an arbitrary sequence, let us consider once
more the blow-up sequence $u_m$ defined in Lemma \ref{density_1}
(iv), and let \[ u_0(\rho,\theta)=\frac{\sqrt{2}}{
3}\rho^{3/2}\max(\cos(3(\theta-\pi/2)/2),0).\] By the proof of
Proposition \ref{2dim}, $\chi_{\{u_m>0\}}$ converges strongly in
$L^1(B_1)$ to $\chi_{\{u_0>0\}}$ along a subsequence. Since this is
true for {\em all} sequences $r_m\searrow 0$, it follows that
\[\chi_{\{x: u(x^0+rx)>0\}}\to \chi_{\{u_0>0\}}\quad\text{strongly in }L^1(B_1)\quad\text{as }r\searrow 0,
\]
which is exactly the first measure estimate. The convergence of
$u_m$ to $u_0$ implies the weak convergence of the sequence of
nonnegative Radon measures $\Delta u_m$ to $\Delta u_0$. Since $u_0$
is
 harmonic in
$\{\cos(3(\theta-\pi/2)/2)>\delta/2\}\cup
\{\cos(3(\theta-\pi/2)/2)<-\delta/2\}$, it follows that \[\Delta
u_m(B_1\cap(\{\cos(3(\theta-\pi/2)/2)>\delta\}\cup
\{\cos(3(\theta-\pi/2)/2)<-\delta\}))\to 0\] as $m\to \infty$. Since
this is true for all sequences $r_m\searrow 0$, the second measure
estimate follows.\end{proof}

\begin{proposition}[Partial regularity in two dimensions]\label{accum}
Let $n=2$, let $u$ be a variational solution of {\rm
(\ref{strongp})},  and suppose that
\[|\nabla u|^2 \le Cx_2^+\quad\text{locally in }\Omega,\]
and that \[r^{-3/2}\int_{B_r(y)} \sqrt{x_2} |\nabla \chi_{\{
u>0\}}|\le C_0\] for all $B_r(y)\subset\subset \Omega$ such that
$y_2=0$. Let $x^0\in S^u$ be a non-degenerate point. Then in some
open neighborhood, $x^0$ is the {\rm only} non-degenerate stagnation
point.
\end{proposition}

\begin{proof} Suppose towards a contradiction that there exists a sequence
$x^m$ of non-degenerate points converging to $x^0$, with $x^m\neq
x_0$ for all $m$. Choosing $r_m := |x^m-x^0|$, there is no loss of
generality in assuming that the sequence $(x^m-x^0)/r^m$ is
constant, with value $z\in\{(-1,0), (1,0)\}$. Consider the blow-up
sequence
\[u_m(x)={u(x^0+r_m x)}/{r_m^{3/2}}.\]
Since $x^m$ is a non-degenerate point for $u$, it follows that $z$
is a non-degenerate point for $u_m$, and therefore Proposition
\ref{isolated} shows that
\[ \Phi^{\textnormal{bound}}_{z,u^m}(0+)=\int_{B_1} x_2^+\chi_{\{
\cos(3(\theta-\pi/2)/2)>0\}}.\] By Lemma \ref{density_1} (v) and the
proof of Proposition \ref{2dim}(ii), the sequence $u_m$ converges
strongly in $W^{1,2}_{\textnormal{loc}}(\R^2)$ to the homogeneous
solution
\[u_0(\rho, \theta)=\frac{\sqrt{2}}{3}
\rho^{3/2}\max(\cos(3(\theta-\pi/2)/2),0),\] where
$x=(\rho\cos\theta, \rho\sin\theta)$, while $\chi_{\{u_m>0\}}$
converges strongly in $L^1_{\textnormal{loc}}(\R^2)$ to
$\chi_{\{u_0>0\}}$. It follows from Lemma \ref{density_2} (v) that
\[ \Phi^{\textnormal{bound}}_{z,u^0}(0+)\geq \limsup_{m\to\infty}
\Phi^{\textnormal{bound}}_{z,u^m}(0+)=\int_{B_1} x_2^+\chi_{\{
\cos(3(\theta-\pi/2)/2)>0\}}\] contradicting the fact that
\[ \Phi^{\textnormal{bound}}_{z,u^0}(0+)=0.\]

\end{proof}

\begin{remark}It follows that in two dimensions $S^u$ can be decomposed into
a countable set of ``Stokes points'' with the asymptotics as in
Proposition \ref{isolated}, accumulating (if at all) only at
``degenerate stagnation points'', and a set of ``degenerate
stagnation points'' which will be analyzed in the following
sections.
\end{remark}

The following Lemma will be used in order to prove the partial
regularity result Proposition \ref{partial}.

\begin{lemma}\label{dist}Let $u$ be a variational solution of {\rm (\ref{strongp})} and suppose that
\[|\nabla u|^2 \le Cx_n^+\quad\text{locally in }\Omega,\] and that
\[r^{1/2-n}\int_{B_r(y)} \sqrt{x_n} |\nabla \chi_{\{ u>0\}}|\le C_0\]
for all $B_r(y)\subset\subset\Omega$ such that $y_n=0$. Suppose that
$x^0\in S^u$ and let $u_0$ be a blow-up limit of
\[u_m(x):={u(x^0+r_m x)}/{r_m^{3/2}}.\] Then for each compact set
$K\subset \R^n$ and each open set $U\supset K\cap
N_{\varsigma,\kappa}^{u_0}$ there exists $m_0<\infty$ such that
$N_{\varsigma,\kappa}^{u_m}\cap K \subset U$ for $m\ge m_0$.
\end{lemma}

\begin{proof} Suppose towards a contradiction that $N_{\varsigma,\kappa}^{u_m}
\cap (K\setminus U)$ contains a sequence $ x^m$ converging to $\bar
x$ as $m \to \infty .$ Then $\bar x_n=0$, and by the locally uniform
Lipschitz continuity of $u_m$, $\bar x\in \{u_0=0\} \cap (K\setminus
U)$. But this contradicts the assumption $U\supset K\cap
N_{\varsigma,\kappa}^{u_0}$ by the uniform convergence of
$u_m$.\end{proof}

\begin{proposition}[Partial regularity in higher dimensions]\label{partial}
Let $u$ be a variational solution of {\rm (\ref{strongp})} and
suppose that
\[|\nabla u|^2 \le Cx_n^+\quad\text{locally in }\Omega,\]
 and that
\[r^{1/2-n}\int_{B_r(y)} \sqrt{x_n}|\nabla \chi_{\{ u>0\}}|\le C_0\] for all
 $B_r(y)\subset\subset \Omega$ such that $y_n=0$. Then the
Hausdorff dimension of the set $\bigcup_{\varsigma,\kappa}
N_{\varsigma,\kappa}^u$ of all non-degenerate points is less or
equal than $n-2$.
\end{proposition}

The proof uses standard tools of geometric measure theory and will
be done in the Appendix.
\begin{remark}
It follows that the Hausdorff dimension of the set of non-degenerate
stagnation points is less or equal than $ n-2$. From Lemma
\ref{equiv} we infer that the set of stagnation points satisfying
the density condition also has dimension at most $n-2$.
\end{remark}

\section{Degenerate Points}

\begin{definition} Let $u$ be a variational solution of {\rm
(\ref{strongp})}. We define \[\Sigma^u := \{ x^0\in S^u:
\Phi^{\textnormal{bound}}_{x^0,u}(0+)= \int_{B_1} x_n^+\}.\]
\end{definition}

\begin{remark} The set $\Sigma^u$ is closed, as a consequence of the upper semicontinuity Lemma
\ref{density_2} (iv).
\end{remark}

\begin{remark}\label{combin}
In the case of two dimensions and a weak solution $u$, we infer from
Lemma \ref{equiv} and Lemma \ref{zero} that the set $S^u \setminus
\Sigma^u$ equals the set of non-degenerate stagnation points and is
according to Proposition \ref{isolated} a finite or countable set.
\end{remark}

The following Lemma is drawn from \cite[Theorem 11.1]{calc}.

\begin{lemma}\label{freq1}
Let $u$ be a variational solution of {\rm (\ref{strongp})}, let
$x^0\in \Sigma^u$, and let $\delta:=\dist(x^0,\partial\Omega)/2$.
Then:

(i) The mean frequency satisfies, for all $r\in(0,\delta)$,
\[r \frac{\int_{B_r(x^0)} |\nabla u|^2}{\int_{\partial
B_r(x^0)}u^2\dh} -   \frac{3}{2}
 \ge r   \frac{\int_{B_r(x^0)} x_n^+(1-\chi_{\{
u>0\}})}{\int_{\partial B_r(x^0)}u^2\dh}\ge 0 .\]

(ii) The function $r\mapsto r^{-n-2}\int_{\partial B_r(x^0)}u^2\dh$
is non-decreasing on $(0,\delta)$ and has the right limit $0$ at
$0.$

(iii) The function $r\mapsto r^{-n-2} \int_{B_r(x^0)}
x_n^+(1-\chi_{\{ u>0\}})$ is integrable on $(0,\delta)$.
\end{lemma}

\begin{proof} (i): The inequality
\[
\Phi^{\textnormal{bound}}_{x^0,u}(0+)\le
\Phi^{\textnormal{bound}}_{x^0,u}(r)\] can be rearranged into
\begin{align}&r^{-n-1} \int_{B_r(x^0)} |\nabla u|^2 - \frac{3}{
2}r^{-n-2} \int_{\partial B_r(x^0)}u^2\dh\non\\ &\quad\ge r^{-n-1}
\int_{B_r(x^0)} x_n^+(1-\chi_{\{ u>0\}})\non,\end{align}
and the right-hand side is clearly nonnegative.

(ii): Plugging in
$\alpha:=-n-2$ into (\ref{easy}) and using (\ref{part2}), it follows
that \begin{align} &\frac{d}{dr} \left(r^{-n-2}  \int_{\partial
B_r(x^0)}u^2\dh
\right)\non\\
&\quad= \frac{2}{r} \left(r^{-n-1} \int_{B_r(x^0)} |\nabla u|^2 -
\frac{3}{2}r^{-n-2}  \int_{\partial B_r(x^0)}u^2\dh \right)\non\\
&\quad\ge 2 r^{-n-2} \int_{B_r(x^0)} x_n^+(1-\chi_{\{
u>0\}}).\non\end{align} Hence $r\mapsto r^{-n-2}  \int_{\partial
B_r(x^0)}u^2\dh$ is non-decreasing on $(0,\delta)$. Using Lemma
\ref{equiv} we obtain that its right limit at $0$ is $0$.

(iii): The above inequality implies that the function $r\mapsto
r^{-n-2} \int_{B_r(x^0)} x_n^+(1-\chi_{\{ u>0\}})$ is in
$L^1((0,\delta)$.
\end{proof}

\section{The Frequency Formula}
\begin{theorem}[Frequency Formula]\label{freq2}
Let $u$ be a variational solution of {\rm (\ref{strongp})}, let
$x^0$ be a point of the closed set $\Sigma^u$, and let
$\delta:=\dist(x^0,\partial\Omega)/2)$. The function
\[F_{x^0,u}(r):= r   \frac{\int_{B_r(x^0)} \Big(|\nabla u|^2+x_n^+
(\chi_{\{ u>0\}}-1)\Big)}{\int_{\partial B_r(x^0)}u^2\dh}
\]
satisfies for a.e. $r\in (0,\delta)$ the identity
\begin{align}&\frac{d}{dr} F_{x^0,u}(r)\non\\
&= \frac{2}{r} \left(\int_{\partial
B_r(x^0)}u^2\dh\right)^{-2}\Bigg[ \int_{\partial B_r(x^0)}(\nabla
u\cdot (x-x^0))^2\dh\int_{\partial
B_r(x^0)}u^2\dh\non\\&\qquad\qquad\qquad\qquad\qquad\qquad\qquad\qquad-\left(\int_{\partial
B_r(x^0)}u\nabla u\cdot (x-x^0)\dh\right)^2\Bigg]\non\\&\quad+
2\frac{\int_{B_r(x^0)} x_n^+(1-\chi_{\{ u>0\}})}{\int_{\partial
B_r(x^0)}u^2\dh} \left(r\frac{\int_{B_r(x^0)} |\nabla
u|^2}{\int_{\partial B_r(x^0)}u^2\dh} -
\frac{3}{2}\right).\non\end{align} The function $r\mapsto
F_{x^0,u}(r)$ is non-decreasing on $(0,\delta)$, and there exists
\[F_{x^0,u}(0+):=\lim_{r\searrow 0}F_{x^0,u}(r)\in [3/2,\infty).\]
\end{theorem}

\begin{remark}
This formula is based on an analogous formula in the interior case
derived by the second author for a more general class of semilinear
elliptic equations (\cite{owr}). The root is the classical frequency
formula of F. Almgren for $Q$-valued harmonic functions
\cite{almgren}. Almgren's formula has subsequently been extended to
various perturbations (see \cite{arshak} for a recent extension).
Note however that while our formula may look like a perturbation of
the ``linear'' formula for $Q$-valued harmonic functions, it is in
fact a truly nonlinear formula. This fact will be become more
obvious in the paper \cite{owr} for more general semilinearities.
\end{remark}

\begin{proof} Assuming the validity of the claimed identity, the
monotonicity of $F_{x^0,u}$ follows from combining the
Cauchy-Schwarz inequality
\begin{align*}
&\int_{\partial B_r(x^0)}(\nabla u\cdot
(x-x^0))^2\dh\int_{\partial B_r(x^0)}u^2\dh \\&\geq\left(\int_{\partial
B_r(x^0)}u\nabla u\cdot (x-x^0)\dh\right)^2
\end{align*}
with Lemma
\ref{freq1} (i). The same Lemma also shows that $r\mapsto
F_{x^0,u}(r)$ is bounded below by $3/2$. Thus it remains to prove
the claimed identity.

Note that \[F_{x^0,u}(r)=\frac{U(r)-\int_{B_1}x_n^+}{W(r)},\] where
$U:=U_{\textnormal{bound}}$ and $W:=W_{\textnormal{bound}}$ are the
functions in the proof of Theorem \ref{elmon2}. Hence
\[\frac{d}{dr} F_{x^0,u}(r)=\frac{U'(r)W(r)-W'(r)(U(r)-r^{-n-1}\int_{B_r(x^0)}x_n^+)}{W^2(r)}.\]
Using (\ref{duu2}) and (\ref{dw2}), it follows that
\begin{align}&\frac{d}{dr} F_{x^0,u}(r)=\left(r^{-n-2}\int_{\partial
B_r(x^0)}u^2\dh\right)^{-2}\left[\left(2 r^{-n-1} \int_{\partial
B_r(x^0)} (\nabla u \cdot \nu)^2\, \dh\right.\right.\non\\&
\qquad\qquad\qquad\left.-3r^{-n-2}\int_{\partial B_r(x^0)} u \nabla
u \cdot \nu
 \dh\right)\left(r^{-n-2}\int_{\partial
B_r(x^0)}u^2\dh\right)\non\\&-\left(r^{-n-1}\int_{B_r(x^0)}
\Big(|\nabla u|^2+x_n^+(\chi_{\{
u>0\}}-1)\Big)\right)\left(2r^{-n-2}\int_{\partial B_r(x^0)} u
\nabla
u\cdot\nu\dh\right.\non\\&\qquad\qquad\qquad\left.\left.-3r^{-n-3}\int_{\partial
B_r(x^0)} u^2\dh\right)\right].\non\end{align}
 Using (\ref{part2}), we obtain
\begin{align}&\frac{d}{dr} F_{x^0,u}(r)\non\\&=\left(\int_{\partial
B_r(x^0)}u^2\dh\right)^{-2}\left[2r \int_{\partial B_r(x^0)} (\nabla
u \cdot \nu)^2\, \dh\int_{\partial B_r(x^0)}u^2\dh\right.\non\\&
\qquad\qquad\qquad\qquad\left.-2r\left(\int_{\partial B_r(x^0)} u
\nabla
u\cdot\nu\dh\right)^2\right.\non\\&+\left(\int_{B_r(x^0)}x_n^+(1-\chi_{\{
u>0\}})\right)\left(2r\int_{\partial B_r(x^0)} u \nabla u\cdot\nu\dh
\left.-3\int_{\partial B_r(x^0)} u^2\dh\right)\right],
\non\end{align} which, upon rearranging and using again
(\ref{part2}) (this time in the reverse direction), gives the
required result.
\end{proof}

\begin{corollary}\label{rightlimit} Let $u$ be a variational solution of {\rm
(\ref{strongp})}, let $x^0$ be a point of the closed set $\Sigma^u$,
and let $\delta:=\dist(x,\partial\Omega)/2)$. Let us consider, for
$r\in(0,\delta)$, the functions
\[D(r) := r\frac{\int_{B_r(x^0)} |\nabla u|^2}{\int_{\partial
B_r(x^0)}u^2\dh}\] and \[V(r) := r\frac{\int_{B_r(x^0)}
x_n^+(1-\chi_{\{ u>0\}})}{\int_{\partial B_r(x^0)}u^2\dh},\] so that
$F_{x_0,u}(r)=D(r)-V(r)$.

 (i) For
every $r\in (0,r)$, the following inequalities hold:
 \[(D-V)'(r) \ge \frac{2}{r} V(r)(D(r)-3/2)\ge\frac{2}{r} V^2(r)
.\]

(ii) The function $r\mapsto \frac{2}{r} V^2(r)$ is integrable on
$(0,\delta)$.
\end{corollary}

\begin{proof}[Proof of Corollary \ref{rightlimit}]
 The inequalities follow from Lemma \ref{freq1} and Theorem
\ref{freq2}. The integrability of $r\mapsto \frac{2}{r}V^2(r)$ is a
consequence of the inequalities.

\end{proof}

\begin{corollary}[Density]\label{density} Let $u$ be a variational
solution of {\rm (\ref{strongp})}. The function \[x\mapsto
F_{x,u}(0+)\] is upper semicontinous on the closed set $\Sigma^u$.

\end{corollary}

\begin{proof}[Proof of Corollary \ref{density}]
 For each $\delta >0$, we have that
\[ F_{x,u}(0+) \leq F_{x,u}(r)\leq F_{x^0,u}(r) +\frac{\delta}{2}\leq
F_{x^0,u}(0+)+\delta,\] if we choose for fixed $x^0\in\Sigma^u$
first $r>0$ and then $|x-x^0|$ small enough.

\end{proof}

The next result is an improvement of Lemma \ref{freq1} at those
points of $\Sigma^u$ at which the frequency is greater than $3/2$.

\begin{lemma}\label{freq3}
Let $u$ be a variational solution of {\rm (\ref{strongp})}, let
$x^0\in \Sigma^u$, and let $\delta:=\dist(x^0,\partial\Omega)/2$.
Suppose that $F_{x^0,u}(0+)>3/2$, and let us denote
$\gamma:=F_{x^0,u}(0+)$ Then:

(i) For all $r\in(0,\delta)$,
\[r \frac{\int_{B_r(x^0)} |\nabla u|^2}{\int_{\partial
B_r(x^0)}u^2\dh} - \gamma
 \ge r   \frac{\int_{B_r(x^0)} x_n^+(1-\chi_{\{
u>0\}})}{\int_{\partial B_r(x^0)}u^2\dh}\ge 0 .\]

(ii) The function $r\mapsto r^{1-n-2\gamma}\int_{\partial
B_r(x^0)}u^2\dh$ is non-decreasing on $(0,\delta)$.

 (iii) The function $r\mapsto
r^{1-n-2\gamma} \int_{B_r(x^0)} x_n^+(1-\chi_{\{ u>0\}})$ is
integrable on $(0,\delta)$.

(iv) For each $\beta\in [0,\gamma)$,
\[\frac{u(x^0+rx)}{r^\beta}\to 0 \quad\text{strongly in }L^2_{\textnormal{loc}}(\R^n)\quad\text{as }r\searrow 0.\]

\end{lemma}

\begin{proof} Part (i) follows from that fact that
$F_{x_0,u}(r)\geq \gamma$ for all $r\in (0,\delta)$. Parts (ii) and
(iii)
 follow by the same arguments as for the corresponding statements in
Lemma \ref{freq1}. It is a consequence of part (ii) that $r\mapsto
r^{-n-2\gamma}\int_{B_r(x^0)}u^2$ is non-decreasing on $(0,\delta)$,
and therefore, for each $\beta\in [0,\gamma)$, \[
r^{-n-2\beta}\int_{B_r(x^0)}u^2\to 0\quad\text{as } r\searrow 0.\]
This implies part (iv) of the Lemma.

\end{proof}

\section{Blow-up limits}
The Frequency Formula allows passing to blow-up limits.
\begin{proposition}\label{blowup} Let $u$ be a variational solution of {\rm
(\ref{strongp})}, and let $x^0\in \Sigma^u$. Then:

 (i)There exist $\lim_{r\searrow 0}V(r)= 0$ and $\lim_{r\searrow 0}D(r)= F_{x^0,u}(0+)$.

(ii) For any sequence $r_m\searrow 0$ as $m\to\infty$, the sequence
\be v_m(x) := \frac{u(x^0+r_m x)}{\sqrt{r_m^{1-n}\int_{\partial
B_{r_m}(x^0)} u^2 \dh}}\label{vm}\ee is bounded in $W^{1,2}(B_1)$.

 (iii) For each sequence $r_m\searrow 0$ as $m\to\infty$ such that the sequence $v_m$ in {\rm (\ref{vm})} converges
 weakly in $W^{1,2}(B_1)$ to a blow-up limit $v_0$, the function $v_0$ is continuous and homogeneous of
degree $F_{x^0,u}(0+)$ in $B_1$,  and satisfies
  $v_0\geq 0$ in $B_1$,
$v_0\equiv 0$ in $B_1\cap\{x_2\leq 0\}$ and  $\int_{\partial
B_1}v_0^2\dh=1$.
\end{proposition}

\begin{proof} The key step in the proof is the statement
(\ref{vol}), which we prove first. We start by writing the
right-hand side of the Frequency Formula in a more convenient form,
using a simple algebraic identity. For any real inner-product space
$(H, \langle,\rangle)$ with norm $||\,||$, any vectors $u$, $v$ and
any scalar $\alpha$, the following holds:
\begin{align}\frac{1}{||u||^4}\Big[||v||^2||u||^2-\langle
v,u\rangle^2\Big]&=\left|\left| \frac{v}{||u||}-\frac{\langle
v,u\rangle}{||u||^2}\frac{u}{||u||}\right|\right|^2\non\\&=\left|\left|
\frac{v}{||u||}-\frac{\langle v,u\rangle
}{||u||^2}\frac{u}{||u||}+\alpha\frac{u}{||u||}\right|\right|^2-\alpha^2,\non\end{align}
where we have used a cancellation due to orthogonality. Using the
notation introduced in Corollary \ref{rightlimit}, we apply the
above identity in the space $L^2(\partial B_r(x^0))$, with $u:=u$,
$v:=r (\nabla u\cdot\nu)$ and $\alpha:= V(r)$, after also taking
into account (\ref{part2}) in the form
\[\frac{\langle v,u\rangle}{||u||^2}= D(r),\]
to obtain from the Frequency Formula that
 \begin{align}\frac{d}{dr} F_{x^0,u}(r)
&= \frac{2}{r} \int_{\partial B_r(x^0)}\left[
\frac{v}{||u||}-D(r)\frac{u}{||u||}+V(r)\frac{u}{||u||}\right]^2\dh\non\\
&\qquad-\frac{2}{r}V^2(r)+\frac{2}{r}V(r)(D(r)-3/2).\non\end{align}
This formula can be rewritten as
 \begin{align} \frac{d}{dr} F_{x^0,u}(r)\label{newmf}
  &= \frac{2}{r} \int_{\partial B_r(x^0)}\Bigg[ \frac{r (\nabla
 u\cdot\nu)}{\left(\int_{\partial B_r(x^0)}u^2\dh\right)^{1/2}}
 \\&-F_{x^0,u}(r)\frac{u}{\left(\int_{\partial
B_r(x^0)}u^2\dh\right)^{1/2}}\Bigg]^2\dh\; +\;
\frac{2}{r} V(r)(F_{x^0,u}(r)-3/2).\non\end{align}
 Since $F_{x^0,u}(r)\geq 3/2$ for all $r\in (0,\delta)$, we obtain therefore
that for all $0<\varrho<\sigma<\delta$,
\begin{align}&\int_\varrho^\sigma\frac{2}{r} \int_{\partial B_r(x^0)}\Bigg[
\frac{r (\nabla u \cdot\nu)}{\left(\int_{\partial
B_r(x^0)}u^2\dh\right)^{1/2}}\label{aff}\\&-F_{x^0,u}(r)\frac{u}{\left(\int_{\partial
B_r(x^0)}u^2\dh\right)^{1/2}}\Bigg]^2\dh\,dr\leq
F_{x^0,u}(\sigma)-F_{x^0,u}(\varrho).\non\end{align}

Let us consider now an arbitrary sequence $r_m$ such that
$r_m\searrow 0$ as $m\to\infty$, and let $v_m$ be the sequence
defined in (\ref{vm}). It follows by scaling from (\ref{aff}) that,
for every $m$ such that $r_m\delta<1$ and for every
$0<\varrho<\sigma<1$,
\begin{align}\int_\varrho^\sigma\frac{2}{r} &\int_{\partial B_r}\left[
\frac{r (\nabla v_m \cdot\nu)}{\left(\int_{\partial
B_r}v_m^2\dh\right)^{1/2}}-F_{x^0,u}(r_mr)\frac{v_m}{\left(\int_{\partial
B_r}v_m^2\dh\right)^{1/2}}\right]^2\dh\,dr\non\\&\leq
F_{x^0,u}(r_m\sigma)-F_{x^0,u}(r_m \varrho)\to 0\quad\text{as
}m\to\infty,\non\end{align} since $F_{x^0,u}$ has a finite limit at
$0$. The above implies that
\begin{align}\int_\varrho^\sigma\frac{2}{r} &\int_{\partial
B_r}\left[ \frac{r (\nabla v_m \cdot\nu)}{\left(\int_{\partial
B_r}v_m^2\dh\right)^{1/2}}-F_{x^0,u}(0+)\frac{v_m}{\left(\int_{\partial
B_r}v_m^2\dh\right)^{1/2}}\right]^2\dh\,dr\non\\&\to 0 \quad\text{as
}m\to\infty.\label{coo}\end{align} Now note that, for every $r\in
(\varrho,\sigma)\subset (0,1)$ and all $m$ as before, it follows by
using Lemma \ref{freq1} that
\[\int_{\partial
B_r}v_m^2\dh=\frac{\int_{\partial B_{r_m r}(x_0)}
u^2\dh}{\int_{\partial B_{r_m}(x_0)} u^2\dh}\leq r^{n+2}\leq 1.\]
Therefore (\ref{coo}) implies that \be\int_{B_\sigma\setminus
B_\varrho} |x|^{-n-3}\left[\nabla v_m(x) \cdot
x-F_{x^0,u}(0+)v_m(x)\right]^2\,dx\to 0\quad\text{as
}m\to\infty.\label{vol}\ee

We can now prove all parts of the Proposition.

(i) Suppose towards a contradiction that (i) is not true. Let
$s_m\to 0$ be a sequence such that $V(s_m)$ is bounded away from
$0$.
 From the integrability of $r\mapsto \frac{2}{r}V^2(r)$ we
obtain that
\[ \min_{r\in [s_m, 2s_m]} V(r)\to 0\quad\text{as }m\to\infty.\]
Let $t_m\in [s_m, 2s_m]$ be such that $V(t_m)\to 0$ as $m\to\infty$.
For the choice $r_m:=t_m$ for each $m$, the sequence $v_m$ given by
(\ref{vm}) satisfies (\ref{vol}). The fact that $V(r_m)\to 0$
implies that $D(r_m)$ is bounded, and hence $v_m$ is bounded in
$W^{1,2}(B_1)$. Let $v_0$ be any weak limit of $v_m$ along a
subsequence. Note that $v_0$ has norm $1$ on $L^2(\partial B_1)$,
since this is true for $v_m$ for all $m$. It follows from
(\ref{vol}) that $v_0$ is homogeneous of degree $F_{x^0,u}(0+)$.
Note that, by using Lemma \ref{freq1} (ii),
\begin{align}
V(s_m)&=\frac{s_m^{-n-1}\int_{B_{s_m}(x^0)}x_n^+(1-\chi_{\{u>0\}})}
{s_m^{-n-2}\int_{\partial B_{s_m}(x^0)}u^2\dh}\non\\&\leq
\frac{s_m^{-n-1}\int_{B_{r_m}(x^0)}x_n^+(1-\chi_{\{u>0\}})}
{(r_m/2)^{-n-2}\int_{\partial B_{r_m/2}(x^0)}u^2\dh}\non\\
&\leq\frac{1}{2}\frac{\int_{\partial
B_{r_m}(x^0)}u^2\dh}{\int_{\partial
B_{r_m/2}(x^0)}u^2\dh}V(r_m)\non\\ &=\frac{1}{2\int_{\partial
B_{1/2}}v_m^2\dh}V(r_m).\label{smrm}\end{align} Since, at least
along a subsequence,
\[\int_{\partial B_{1/2}}v_m^2\dh\to\int_{\partial B_{1/2}}v_0^2\dh>0,\]
(\ref{smrm}) leads to a contradiction. It follows that indeed
$V(r)\to 0$ as $r\searrow 0$. This implies  that $D(r)\to
F_{x^0,u}(0+)$.

(ii) Let $r_m$ be an arbitrary sequence with $r_m\searrow 0$. The
boundedness of the sequence $v_m$ in $W^{1,2}(B_1)$ is equivalent to
the boundedness of $D(r_m)$, which is true by (i).

(iii) Let $r_m\searrow 0$ be an arbitrary sequence such that $v_m$
converges weakly to $v_0$. The homogeneity degree $F_{x^0,u}(0+)$ of
$v_0$ follows directly from (\ref{vol}). The homogeneity of $v_0$,
together with the fact that $v_0$ belongs to $W^{1,2}(B_1)$, imply
that $v_0$ is continuous.
 The fact that $\int_{\partial B_1}v_0^2\dh=1$ is a consequence of
 $\int_{\partial B_1}v_m^2\dh=1$ for all $m$,
 and the remaining claims of the Proposition are obvious.
\end{proof}

\section{Concentration compactness in two dimensions}
In the two-dimensional case we prove concentration compactness which
allows us to preserve variational solutions in the blow-up limit at
degenerate points and excludes concentration. In order to do so we
combine the concentration compactness result of Evans and M\"uller
\cite{evansmueller} with information gained by our Frequency
Formula. In addition, we obtain strong convergence of our blow-up
sequence which is necessary in order to prove our main theorems. The
question whether the following Theorem holds in any dimension seems
to be a hard one.

\begin{theorem}\label{comp}
Let $n=2$, let $u$ be a variational solution of
{\rm(\ref{strongp})}, and let $x^0\in\Sigma^u$. Let $r_m\searrow 0$
be such that the sequence $v_m$ given by {\rm(\ref{vm})} converges
weakly to $v_0$ in $W^{1,2}(B_1)$. Then $v_m$ converges to $v_0$
strongly in $W^{1,2}_{\textnormal{loc}}(B_1\setminus \{ 0\})$, and
$v_0$ satisfies $v_0 \Delta v_0=0$ in the sense of Radon measures on
$B_1$.
\end{theorem}

\begin{proof}

Note first that, since  $v_0$ is by Proposition \ref{blowup} a
nonnegative continuous function, $v_0\Delta v_0$ is well defined as
a nonnegative Radon measure on $B_1$.

 Let $\sigma$ and $\varrho$ with $0<\varrho<\sigma<1$ be arbitrary.
We know that $\Delta v_m \ge 0$ and $\Delta
v_m(B_{(\sigma+1)/2})\leq C_1$ for all $m$. In order to apply the
concentrated compactness result \cite{evansmueller}, we regularize
each $v_m$ to \[\tilde v_m := v_m*\phi_m \in C^\infty(B_1),\] where
$\phi_m$ is a standard mollifier such that \[\Delta \tilde v_m \ge
0, \int_{B_\sigma} \Delta \tilde v_m\le C_2<+\infty \quad\text{for
all } m,\] and \[\Vert v_m-\tilde v_m\Vert_{W^{1,2}(B_\sigma)}\to
0\quad\text{as }m\to\infty.\] From \cite[Chapter 4, Theorem
3]{evans} we know that $\nabla \tilde v_m$ converges a.e. to the
weak limit $\nabla v_0$, and the only possible problem is
concentration of $|\nabla \tilde v_m|^2$. By \cite[Theorem
1.1]{evansmueller} and \cite[Theorem 3.1]{evansmueller} we obtain
that \[\partial_1 \tilde v_m \partial_2 \tilde v_m \to
 \partial_1 v_0 \partial_2 v_0\]
and \[(\partial_1 \tilde v_m)^2- (\partial_2 \tilde v_m)^2 \to
(\partial_1 v_0)^2- (\partial_2 v_0)^2\] in the sense of
distributions on $B_\sigma$ as $m\to\infty$. It follows that
\be\label{evm}\partial_1 v_m \partial_2 v_m \to
 \partial_1 v_0 \partial_2 v_0\ee
and \[(\partial_1 v_m)^2- (\partial_2 v_m)^2 \to (\partial_1 v_0)^2-
(\partial_2 v_0)^2\] in the sense of distributions on $B_\sigma$ as
$m\to\infty$. Let us remark that this alone would allow us to pass
to the limit in the domain variation formula for $v_m$ in the set
$\{x_2>0\}$.

Observe now that (\ref{vol}) shows that \[\nabla v_m(x)\cdot x -
F_{x^0,u}(0+) v_m(x)\to 0\] strongly in $L^2(B_\sigma\setminus
B_\varrho)$ as $m\to\infty$. It follows that \[\partial_1 v_m x_1 +
\partial_2 v_m x_2\to \partial_1 v_0 x_1 + \partial_2 v_0 x_2\]
 strongly in $L^2(B_\sigma\setminus B_\varrho)$ as $m\to\infty$.
But then \begin{align*}
&\int_{B_\sigma\setminus B_\varrho} (\partial_1 v_m
\partial_1 v_m x_1 +
 \partial_1 v_m \partial_2 v_m x_2)\eta(x)
 \,dx
\\&\to \int_{B_\sigma\setminus B_\varrho} (\partial_1 v_0 \partial_1
v_0 x_1 + \partial_1 v_0 \partial_2 v_0 x_2)\eta(x)
 \,dx\end{align*}
for each $\eta \in C^0_0(B_\sigma\setminus \overline B_\varrho)$ as
$m\to\infty$. Using (\ref{evm}), we obtain that
\[\int_{B_\sigma\setminus B_\varrho} (\partial_1 v_m)^2 x_1 \eta(x)
 \,dx
\to \int_{B_\sigma\setminus B_\varrho} (\partial_1 v_0)^2 x_1
\eta(x)
 \,dx\]
for each $0\le \eta \in C^0_0((B_\sigma\setminus \overline
B_\varrho)\cap \{ x_1>0\})$ and for each $0\ge \eta \in
C^0_0((B_\sigma\setminus \overline B_\varrho)\cap \{ x_1 <0\})$
 as $m\to\infty$.
Repeating the above procedure three times for rotated sequences of
solutions (by $45$ degrees) yields that $\nabla v_m$ converges
strongly in $L^2_{\textnormal{loc}}(B_\sigma\setminus \overline
B_\varrho)$. Since $\sigma$ and $\varrho$ with $0<\varrho<\sigma<1$
were arbitrary, it follows that $\nabla v_m$ converges to $\nabla
v_0$ strongly in $L^2_{\textnormal{loc}}(B_1\setminus \{ 0\})$.

As a consequence of the strong convergence, we see that \[\int_{B_1}
\nabla (\eta v_0)\cdot \nabla v_0 =0 \quad\text{ for all } \eta \in
C^1_0(B_1\setminus\{0\}).\] Combined with the fact that $v_0=0$ in
$B_1\cap\{x_2\leq 0\}$, this proves that $v_0\Delta v_0=0$ in the
sense of Radon measures on $B_1$.
 \end{proof}

\section{Degenerate points in two dimensions}

\begin{theorem}\label{deg2d}Let $n=2$ and let $u$ be a variational
 solution of {\rm (\ref{strongp})}. Then at each point $x^0$ of the set
$\Sigma^u$ there exists an integer $N(x^0)\ge 2$ such that
\[F_{x^0,u}(0+)=N(x^0)\]
and
\[\frac{u(x^0+rx)}{\sqrt{r^{-1}\int_{\partial B_{r}(x^0)} u^2
\dhone}}\to
\frac{\rho^{N(x^0)}|\sin(N(x^0)\theta)|}{\sqrt{\int_0^{2\pi}\sin^2(N(x^0)\theta)
 d\theta}}\quad\text{as }r\searrow 0,\]
strongly in $W^{1,2}_{\textnormal{loc}}(B_1\setminus\{0\})$ and
weakly in $W^{1,2}(B_1)$, where $x=(\rho\cos\theta,\rho\sin\theta)$.
\end{theorem}

\begin{proof} Let $r_m\searrow 0$ be an arbitrary sequence such that the sequence $v_m$
given by (\ref{vm}) converges weakly in $W^{1,2}(B_1)$ to a limit
$v_0$. By Proposition \ref{blowup} (iii) and Theorem \ref{comp},
$v_0\not\equiv 0$, $v_0$ is homogeneous of degree $F_{x^0,u}(0+)\ge
3/2$, $v_0$ is continuous, $v_0\ge 0$  and $v_0 \equiv 0$ in $\{
x_2\leq 0\}$, $v_0\Delta v_0=0$ in $B_1$ as a Radon measure, and the
convergence of $v_m$ to $v_0$ is strong in
$W^{1,2}_{\textnormal{loc}}(B_1\setminus\{0\})$.
 Moreover, the strong
convergence of $v_m$ and the fact proved in Proposition \ref{blowup}
(i) that $V(r_m)\to 0$ as $m\to \infty$ imply that \[ 0=\int_{B_1}
\Big( {\vert \nabla v_0 \vert}^2 \div\phi   -   2 \nabla v_0   D\phi
\nabla v_0\Big)\] for every $\phi\in C^1_0(B_1\cap \{ x_2
>0\};\R^2).$ It follows that at each polar coordinate point
$(1,\theta)\in \partial B_1 \cap \partial \{ v_0>0\}$, \[\lim_{\tau
\searrow \theta} \partial_\theta v_0(1,\tau) = -  \lim_{\tau
\nearrow \theta} \partial_\theta v_0(1,\tau).\] Computing the
solution of the ODE on $\partial B_1$, using the homogeneity of
degree $F_{x^0,u}(0+)$ of $v_0$ and the fact that $\int_{\partial
B_1}v_0^2\dhone=1$ , yields that $F_{x^0,u}(0+)$ must be an {\em
integer} $N(x^0)\geq 2$ and that \be
v_0(\rho,\theta)=\frac{\rho^{N(x^0)}|\sin(N(x^0)\theta)|}{\sqrt{\int_0^{2\pi}\sin^2(N(x^0)\theta)
 d\theta}}\label{vo}.\ee
The desired conclusion follows from Proposition \ref{blowup} (ii).
\end{proof}

\begin{theorem}\label{degisol}
Let $n=2$ and let $u$ be a variational solution of {\rm
(\ref{strongp})}. Then the set $\Sigma^u$ is locally in $\Omega$ a
finite set.
\end{theorem}

\begin{proof} Suppose towards a contradiction that there is a sequence of
points $x^m\in\Sigma^u$ converging to  $x^0\in \Omega$, with
$x^m\neq x^0$ for all $m$. From the upper semicontinuity Lemma
\ref{density_2} (iv) we infer that $x^0\in \Sigma^u$. Choosing
$r_m:=2|x^m-x^0|$, there is no loss of generality in assuming that
the sequence $(x^m-x^0)/r_m$ is constant, with value $z\in
\{(-1/2,0), (1/2,0)\}$. Consider the blow-up sequence $v_m$ given by
(\ref{vm}), and also the sequence
\[u_m(x)=u(x^0+r_m x)/r_m^{3/2}.\] Note that each $u_m$ is a variational
solution of (\ref{strongp}), and $v_m$ is a scalar multiple of
$u_m$. Since $x_m\in\Sigma^u$, it follows that $z\in\Sigma^{u_m}$.
Therefore, Lemma \ref{freq1} shows that, for each $m$,
\[\int_{B_r(z)}|\nabla v_m|^2\geq \frac{3}{2}\int_{\partial B_r(z)}v_m^2\dhone\quad\text{for all }r\in (0,1/2).
\]
It is a consequence of Theorem \ref{deg2d} that the sequence $v_m$
converges strongly in $W^{1,2}(B_{1/4}(z))$ to $v_0$ given by
(\ref{vo}), hence
\[\int_{B_r(z)}|\nabla v_0|^2\geq \frac{3}{2}\int_{\partial B_r(z)}v_0^2\dhone\quad\text{for all }r\in (0,1/4).\]
But this contradicts the fact, which can be checked directly, that
\[\lim_{r\searrow 0} r\frac{\int_{B_r(z)} |\nabla v_0|^2}{\int_{\partial
B_r(z)}v_0^2\dhone}=1.\]
\end{proof}

\section{Conclusion}

\begin{theorem}\label{main1}
Let $n=2$, let $u$ be a weak solution of {\rm (\ref{strongp})}, and
suppose that
\[|\nabla u|^2 \leq x_2^+\quad\text{in }\Omega.\]
 Then the set $S^u$ of stagnation points is a finite or countable set. Each
accumulation point of $S^u$ is a point of the locally finite set
$\Sigma^u$.

 At each point $x^0$ of $S^u\setminus\Sigma^u$,
\[\frac{u(x^0+rx)}{{r^{3/2}}}\to \frac{\sqrt{2}}{3}
\rho^{3/2}\max(\cos(3(\theta-\pi/2)/2),0) \quad\text{as } r\searrow
0,\]
 strongly in $W^{1,2}_{\textnormal{loc}}(\R^2)$ and locally
uniformly on $\R^2$, where $x=(\rho\cos\theta, \rho\sin\theta)$.
Moreover,
\[\mathcal{L}^2(B_1\cap (\{x:
u(x^0+rx)>0\}\bigtriangleup\{\cos(3(\theta-\pi/2)/2)>0\}))\to
0\quad\text{as }r\searrow 0,\] and, for each $\delta>0$,
\[r^{-3/2}\Delta
u((x^0+B_r)\cap(\{\cos(3(\theta-\pi/2)/2)>\delta\}\cup
\{\cos(3(\theta-\pi/2)/2)<-\delta\}))\to 0\] as $r\searrow 0$.

At each point $x^0$ of $\Sigma^u$ there exists an integer $N(x^0)\ge
2$ such that \[\frac{u(x^0+rx)}{{r^\beta}}\to 0 \quad\text{as
}r\searrow 0,\] strongly in $L^2_{\textnormal{loc}}(\R^2)$ for each
$\beta \in [0,N(x^0))$, and
\[\frac{u(x^0+rx)}{\sqrt{r^{-1}\int_{\partial B_r(x^0)} u^2 \dhone}}
\to
\frac{\rho^{N(x^0)}|\sin(N(x^0)\theta)|}{\sqrt{\int_0^{2\pi}\sin^2(N(x^0)\theta)
 d\theta}} \quad\text{as } r\searrow 0,\]
strongly in $W^{1,2}_{\textnormal{loc}}(B_1\setminus\{0\})$ and
weakly in $W^{1,2}(B_1)$, where $x=(\rho\cos\theta,\rho\sin\theta)$.
\end{theorem}

\begin{proof} By Lemma \ref{arevar}, $u$ is a variational solution of {\rm (\ref{strongp})}
and satisfies
\[r^{-3/2}\int_{B_r(y)} \sqrt{x_2} |\nabla \chi_{\{ u>0\}}|\le C_0\]
for all $B_r(y)\subset\subset \Omega$ such that $y_n=0$. Combining
Proposition \ref{isolated}, Lemma \ref{equiv}, Lemma \ref{zero},
Proposition \ref{accum}, Lemma \ref{freq3}, Theorem \ref{degisol}
and Theorem \ref{deg2d}, we obtain that the set $S^u$ is a finite or
countable set with asymptotics as in the statement, and that the
only possible accumulation points are elements of
$\Sigma^u$.\end{proof}

\begin{theorem}\label{main2}
Let $n=2$, let $u$ be a weak solution of {\rm (\ref{strongp})}, and
suppose that
\[|\nabla u|^2 \leq x_2^+\quad\text{in }\Omega.\]
 Suppose moreover that $\{ u=0\}$ has locally only finitely many
connected components. Then the set $S^u$ of stagnation points is
locally in $\Omega$ a finite set. At each stagnation point $x^0$,
\[\frac{u(x^0+rx)}{{r^{3/2}}}\to \frac{\sqrt{2}}{3}
\rho^{3/2}\max(\cos(3(\theta-\pi/2)/2),0) \quad\text{as } r\searrow
0,\]
 strongly in $W^{1,2}_{\textnormal{loc}}(\R^2)$ and locally
uniformly on $\R^2$, where $x=(\rho\cos\theta, \rho\sin\theta)$, and
in an open neighborhood of $x^0$ the topological free boundary
$\partial\{u>0\}$ is the union of two $C^1$-graphs with right and
left tangents at $x^0$.
\end{theorem}

\begin{proof} We first show that the set
$\Sigma^u$ is empty. Suppose towards a contradiction that there
exists $x^0\in \Sigma^u$. From Theorem \ref{deg2d} we infer that
there exists an integer $N(x^0)\ge 2$ such that
\[\frac{u(x^0+rx)}{\sqrt{r^{-1}\int_{\partial B_r(x^0)} u^2 \dhone}}
\to \frac{|\sin(N(x^0)\theta)|}{\sqrt{\int_0^{2\pi}\sin^2(N(x^0)
d\theta)  d\theta}} \quad\text{as } r\searrow 0,\] strongly in
$W^{1,2}_{\textnormal{loc}}(B_1\setminus\{0\})$ and weakly in
$W^{1,2}(B_1)$, where $x=(\rho\cos\theta,\rho\sin\theta)$. But then
the assumption about connected components implies that
$\partial_{\textnormal{red}} \{ x: u(x^0+rx)>0\}$ contains the image
of a continuous curve converging, as $r\searrow 0$, locally in $\{
x_2>0\}$ to a half-line $\{ \alpha z : \alpha
>0\}$ where $z_2>0$. It follows that
\[\mathcal{H}^1(\{ x_2>1/2\}\cap \partial_{\textnormal{red}} \{ x:u(x^0+rx)>0\})
\ge c_1 >0,\] contradicting \[ 0 \gets \Delta
\frac{u(x^0+rx)}{{r^{3/2}}}(B_1) = \int_{B_1\cap
\partial_{\textnormal{red}} \{ x:u(x^0+rx)>0\}}\sqrt{x_2}\,\dhone.\]
Hence $\Sigma^u$ is indeed empty.

Let $x^0\in S^u$. Theorem \ref{main1} shows that
\[\frac{u(x^0+rx)}{{r^{3/2}}}\to \frac{\sqrt{2}}{3}
\rho^{3/2}\max(\cos(3(\theta-\pi/2)/2),0) \quad\text{as } r\searrow
0,\]
 strongly in $W^{1,2}_{\textnormal{loc}}(\R^2)$ and locally
uniformly on $\R^2$, where $x=(\rho\cos\theta, \rho\sin\theta)$. To
prove the last statement we now use flatness-implies-regularity
results in the vein of \cite[Theorem 8.1]{AC}. More precisely, for
each $\sigma\le \sigma_0$ and $y^0\in B_\delta(x^0)\cap \partial\{
u>0\} \cap \{ y^0_1<x^0_1\}$, $u \in F(\sigma,0;\sigma_0\sigma^2)$
in $B_{r/2}(y^0)$ in the direction $\eta=(1/2,-\sqrt{3}/2)$ (cf.
\cite[Definition 4.1]{aw}) provided that $\delta$ has been chosen
small enough, meaning that $u$ is a weak solution and satisfies
\[ u(x)=0 \quad\text{in } \{ x\in
B_r(y^0):x\cdot \eta\ge \sigma r\}\] and \[|\nabla u|\le
\sqrt{y^0_2} (1+\sigma_0\sigma^2)\quad\text{ in } B_r(y^0).\] From
the proof of \cite[Theorem 8.4]{aw} (with the proviso that the
parabolic monotonicity formula in \cite{aw} is replaced by the local
elliptic formula in Theorem \ref{elmon2} (i) and the solution has
been extended to a constant function of the time variable) we infer
that $B_{r/2}(x^0)\cap \partial\{ u>0\} \cap \{ y^0_1<x^0_1\}$ is
the graph of a $C^{1,\alpha}$-function and that the outer normal
$\nu$ satisfies $|\nu(y^0)-\eta|\le \sigma$. It follows that
$B_\delta(x^0)\cap \partial\{ u>0\} \cap \{ y^0_1\le x^0_1\}$ is the
graph of a $C^1$-function. The same holds for $B_\delta(x^0)\cap
\partial\{ u>0\} \cap \{ y^0_1\ge x^0_1\}$.
\end{proof}

\section{Appendix}
\begin{proof}[Proof of Lemma \ref{arevar}.] For any $\phi \in
C^1_0(\Omega\cap \{ x_n>\tau\};\R^n)$ and small positive $\delta$ we
find a covering
\[ \bigcup_{i=1}^\infty B_{r_i}(x^i) \supset \text{supp }\phi
\cap (\partial \{u>0\} \setminus
\partial_{\textnormal{red}} \{u>0\}) \] satisfying
$\sum_{i=1}^\infty {r_i}^{n-1} \le \delta$, and by the fact that
$\text{supp }\phi \cap (\partial \{u>0\} \setminus
\partial_{\textnormal{red}} \{u>0\})$ is a compact set we may pass to a finite
subcovering
\[ \bigcup_{i=1}^{N_\delta} B_{r_i}(x^i) \supset \text{supp }\phi
\cap (\partial \{u>0\} \setminus
\partial_{\textnormal{red}} \{u>0\})\] satisfying
$\sum_{i=1}^{N_\delta} {r_i}^{n-1} \le \delta$. 

We also know that $u
\in C^1(\overline{\{u>0\}} \cap (\text{supp }\phi \setminus
\bigcup_{i=1}^{N_\delta} B_{r_i}(x^i)))$ and that $u$ satisfies the
free-boundary condition \[{\vert \nabla u \vert}^2 = x_n\quad\text{
on }
\partial_{\textnormal{red}}\{u>0\} \cap(\text{supp }\phi \setminus
\bigcup_{i=1}^{N_\delta} B_{r_i}(x^i)).\] Formally integrating by
parts in $\{u>0\} \setminus \bigcup_{i=1}^{N_\delta} B_{r_i}(x^i)$
(this can be justified rigorously approximating
$\bigcup_{i=1}^{N_\delta} B_{r_i}(x^i)$ from above by $A_\epsilon$
such that $\partial (\{u>0\} \setminus A_\epsilon)$ is locally in
$\text{supp }\phi$ a $C^3$-surface) we therefore obtain
\begin{align} &\left\vert \int_\Omega \left({\vert \nabla u \vert}^2
\div\phi  -  2  \nabla u D\phi  \nabla u +   x_n \chi_{\{u>0\}}
\div\phi+ \chi_{\{ u>0\}} \phi_n\right)\right\vert\non\\& \le
\left\vert \int_{\bigcup_{i=1}^{N_\delta} B_{r_i}(x^i)}\left( {\vert
\nabla u \vert}^2  \div\phi  -  2  \nabla u  D\phi \nabla u + x_n
\chi_{\{u>0\}}  \div\phi+   \chi_{\{ u>0\}} \phi_n\right)
\right\vert\non\\&\qquad +  \left\vert \int_{\partial
(\bigcup_{i=1}^{N_\delta} B_{r_i}(x^i)) \cap \{u>0\}} \left( {\vert
\nabla u \vert}^2 \phi\cdot \nu  -  2 \nabla u \cdot \nu  \nabla u
\cdot \phi   +  x_n \phi\cdot \nu \right) d{\mathcal H}^{n-1}
\right\vert\non\\&\qquad +  \left\vert \int_{\partial \{u>0\}
\setminus (\bigcup_{i=1}^{N_\delta} B_{r_i}(x^i))} (x_n - {\vert
\nabla u \vert}^2)  \phi\cdot \nu  d{\mathcal H}^{n-1}
\right\vert\non\\&  \le  C_1  \sum_{i=1}^{N_\delta} {r_i}^n  + C_2
 \sum_{i=1}^{N_\delta} {r_i}^{n-1}+ 0 ,\non\end{align} and letting
$\delta \to 0 ,$ we realize that $u$ is a variational solution of
{\rm (\ref{strongp})} in the set $\Omega \cap \{ x_n>\tau\}$. Note
that the above extends to Lipschitz functions $\phi$. Next, let us
take $\phi \in C^1_0(\Omega;\R^n)$ and $\eta:= \min(1,x_n/\tau)$,
plug in the product $\eta \phi$ into the already obtained result,
and use the assumption $|\nabla u|^2 \le Cx_n^+$:
\begin{align}0&=\int_\Omega \eta \Big( {\vert \nabla u \vert}^2 \div\phi
- 2 \nabla u D\phi\nabla u +x_n   \chi_{\{ u>0\}} \div\phi +
\chi_{\{ u>0\}} \phi_n\Big)\non\\&\qquad + \frac{1}{\tau}
\int_{\Omega\cap\{ 0<x_n<\tau\}} \phi\cdot \Big ({\vert \nabla u
\vert}^2 e_n - 2
\partial_n u \nabla u + x_n  \chi_{\{ u>0\}}e_n\Big)\non\\& =
 o(1) + \int_\Omega \Big( {\vert \nabla u \vert}^2 \div\phi
- 2 \nabla u   D\phi \nabla u + x_n \chi_{\{ u>0\}} \div\phi
  + \chi_{\{ u>0\}} \phi_n\Big)\quad\text{as $\tau\to 0$}.\non\end{align}

Last, let us prove that \[r^{1/2-n}\int_{B_r(y)} \sqrt{x_n} |\nabla
\chi_{\{ u>0\}}|\leq C_0\] for all $B_r(y)\subset\subset \Omega$
such that $y_n=0$. Let us consider such an $y$, and the family of
scaled functions
 \[u_r(x):= u(y+rx)/r^{3/2}.\]
Using the assumption
\[|\nabla u|^2 \le Cx_n^+\quad\text{locally in }\Omega\]
and the weak solution property, it follows that
\begin{align} C_0 &\ge \int_{\partial B_1} \nabla u_r(x)\cdot x \dh
\non \; = \; \Delta u_r(B_1)\\&= \int_{B_1\cap
\partial_{\textnormal{red}} \{ u_r>0\}} \sqrt{x_n}\dh\; = \; r^{1/2-n}
\int_{B_r(y)\cap \partial_{\textnormal{red}} \{ u>0\}}
\sqrt{x_n}\dh,\non\end{align} as required.
\end{proof}

\begin{proof}[Proof of Proposition \ref{partial}.] The proof is a standard
dimension reduction argument following \cite[Section 11]{giu}. In
each step, blowing up once transforms the free boundary into a cone,
blowing up a second time at a point different from the origin
transforms the free boundary into a cylinder, and passing to a
codimension $1$ cylinder section reduces the dimension of the whole
problem by $1$.

Let us do this in some more detail: Suppose that there exists
$s>n-2$ and $\varsigma>0$, $\kappa>0$ such that ${\mathcal
H}^s(N_{\varsigma,\kappa}^u)
>0 .$ Then we may use \cite[Proposition 11.3]{giu}, Lemma \ref{dist}
as well as \cite[Lemma 11.5]{giu} at ${\mathcal H}^s$-a.e. point of
$N_{\varsigma,\kappa}^u$ to obtain a blow-up limit $u_0$ satisfying
${\mathcal H}^{s,\infty}(N_{\varsigma,\kappa}^{u_0})>0$. According
to Lemma \ref{density_1} $u_0$ is a homogeneous variational solution
on $\R^n$, where $\chi_{\{ u>0\}}$ has to be replaced by $\chi_0:=
\lim_{m\to\infty} \chi_{\{ u_m>0\}}$ in Definition \ref{vardef}. We
proceed with the dimension reduction: By \cite[Lemma 11.2]{giu} we
find a point $\bar x \in N_{\varsigma,\kappa}^{u_0}\setminus\{0\}$
at which the density in \cite[Proposition 11.3]{giu} is estimated
from below. Now each blow-up limit $u_{00}$ with respect to $\bar x$
(and with respect to a subsequence $m\to \infty$ such that the limit
superior in \cite[Proposition 11.3]{giu} becomes a limit) again
satisfies the assumptions of Lemma \ref{density_1}. In addition, we
obtain from the homogeneity of $u_{00}$ as in Lemma 3.1 of
\cite{jga} that the rotated $u_{00}$ is constant in the direction of
the $n$-th unit vector. Defining $\bar u$ as the restriction of this
rotated solution to $\R^{n-1} ,$ it follows therefore that
${\mathcal H}^{s-1}(N_{\varsigma,\kappa}^{\bar u})>0$. Repeating the
whole procedure $n-2$ times we obtain a nontrivial homogeneous
solution $u^\star$ in $\R^2$, satisfying ${\mathcal
H}^{s-(n-2)}(N_{\varsigma,\kappa}^{u^\star})>0 ,$ by Proposition
\ref{2dim} a contradiction. \end{proof}

\bibliographystyle{siam}
\bibliography{varvaruca_weiss}

\end{document}